\documentclass{amsart}

\usepackage[all]{xy}

\usepackage{amsmath,amssymb,tabularray,verbatim}
\usepackage{color}

\numberwithin{equation}{section}

\theoremstyle{plain}
\newtheorem{theorem}{Theorem}[section]
\newtheorem{corollary}[theorem]{Corollary}
\newtheorem{lemma}[theorem]{Lemma}
\newtheorem{proposition}[theorem]{Proposition}
\theoremstyle{remark}
\newtheorem{remark}[theorem]{Remark}
\theoremstyle{definition}
\newtheorem{definition}[theorem]{Definition}

\makeatletter

\def\R{\mathbb{R}}
\def\C{\mathbb{C}}
\def\Z{\mathbb{Z}}

\newcommand{\lie}[1]{\mathfrak{#1}}
\newcommand{\lier}[1]{\mathfrak{#1}_{0}}

\newcommand\stWh[3]{I_{#1,#2}^{#3}}
\newcommand\stWhmg[3]{I_{#1,#2}^{#3}}
\def\WhL{\stWh{\eta}{\Lambda}{\circ}}
\def\WhLmg{\stWhmg{\eta}{\Lambda}{}}
\def\WhLast{\stWh{\eta}{\Lambda}{\ast}}
\newcommand\indWh[2]{Y_{#1,#2}}
\def\indWhL{\indWh{\Lambda}{\eta}}
\newcommand\duWh[2]{\Pi_{\lie{g},\{e\}}^{\gK}(Y_{#1,#2})}
\def\duWhL{\duWh{\Lambda}{\eta}}
\def\gK{\lie{g},K}
\def\brgK{(\gK)}
\def\kK{\lie{k},K}
\def\Ug{U(\lie{g})}
\def\Zg{Z(\lie{g})}
\def\Un{U(\lie{n})}
\def\Uk{U(\lie{k})}

\def\soc{\mathrm{soc}\,}

\def\Hom{\mathrm{Hom}}
\def\Ext{\mathrm{Ext}}

\def\Obj{\mathrm{Obj}}
\def\Mor{\mathrm{Mor}}

\def\tr{\mathrm{tr}}

\def\Ind{\mathrm{Ind}}
\def\Ad{\mathrm{Ad}}
\def\ad{\mathrm{ad}}

\def\pro{\mathrm{pro}}
\def\res{\mathrm{res}}
\def\ch{\mathrm{ch}}

\def\Dim{\mathrm{Dim}}

\def\WF{\mathrm{WF}}
\def\Wh{\mathrm{Wh}}

\def\Real{\mathrm{Re}}

\allowdisplaybreaks[3]

\title[Self-Adjoint Properties of Whittaker modules]
{On the Self-adjoint properties of the standard Whittaker
  $\brgK$-modules}
\author{Kenji Taniguchi}
\address{
Department of Mathematics, 
Aoyama Gakuin University, 
5-10-1, Fuchinobe, Chuo-ku, Sagamihara, Kanagawa 252-5258, Japan. }
\email{taniken@math.aoyama.ac.jp}

\begin{document}

\begin{abstract}
The structure of the standard Whittaker $\brgK$-module is examined in
the case when the group in question is a real split reductive linear Lie group. 
This module is an injective object in the category of Harish-Chandra
$\brgK$-modules which admit a fixed infinitesimal character. 
The global character of this module is determined. 
The main theorem of this paper is that it has a self-adjoint structure. 
Also obtained are the explicit socle filtrations of the standard
Whittaker $\brgK$-modules for the rank two split groups $SL(3,\R)$,
$Sp(2,\R)$ and $G_{2}$(split). 
\end{abstract}

\maketitle

\section{Introduction}
\label{section:introduction}

Let $G$ be a real reductive linear Lie group in the sense 
of \cite{Vogan Green} 
and $G=KAN$ be an Iwasawa decomposition of it. 
Let $\eta: N \longrightarrow \C^{\times}$ be a unitary character of
$N$. 
We assume $\eta$ is non-degenerate, i.e. it is
non-trivial on every root space corresponding to a simple root of
$\Sigma(\lier{n}, \lier{a})$. 
Define 
\begin{equation}\label{eq:space of Whittaker functions}
C^{\infty}(G/N; \eta) 
:= \{f: G \overset{C^{\infty}}{\longrightarrow} \C \,|\, 
f(gn) = \eta(n)^{-1} f(g), \enskip g \in G, n \in N\}
\end{equation}
and call it the space of Whittaker functions on $G$. 
This is a representation space of
$G$ by the left translation, which is denoted by $\ell$. 
Let $C^{\infty}(G/N;\eta)_{K}$ be the subspace of
$C^{\infty}(G/N;\eta)$ consisting of $K$-finite vectors.

As for the subrepresentations of this space, 
many deep and interesting results are known, 
called the theory of Whittaker models. 
For example, an irreducible $\brgK$-module $X$ is a submodule of 
$C^{\infty}(G/N;\eta)_{K}$ if and only if $X$ is quasi-large 
(cf. Theorem~\ref{theorem:irred sub of WhL}). 
Here, $X$ is called quasi-large if the Gelfand-Kirillov dimension 
(\cite{Vogan GK-dim}) of it is equal to the dimension of $N$. 
If the group $G$ is quasi-split, then $X$ is quasi-large if and only
if it is {\it large} in the sense of \cite{Vogan GK-dim}. 
In this case, $X$ is large if and only if the $\tau$-invariant 
of it is empty. 

On the other hand, there are not so many results on the structure of
the whole space. 
Since the space $C^{\infty}(G/N; \eta)_{K}$ is too large to analyze, 
we need to cut off a submodule of suitable size from it. 
Consider the subspace of $C^{\infty}(G/N; \eta)_{K}$ consisting of
those functions $f$ which satisfy the following conditions: 
\begin{enumerate}
\item
$f$ is a joint eigenfunction of $\Zg$ (the center of
  the universal enveloping algebra $\Ug$) 
with infinitesimal character  
$\chi_{\Lambda}$: 
$\ell(z) f = \chi_{\Lambda}(z) f$, $z \in \Zg$. 
\item
$f$ is of moderate growth (\cite{W}).  
\end{enumerate}
Denote by $\WhL$ and $\WhLmg$ the subspaces consisting of 
$f \in C^{\infty}(G/N; \eta)_{K}$ satisfying (1) and (1), (2),
respectively. 
If we want to specify the group $G$, we also denote them by 
$\stWh{\eta}{\Lambda}{G}$ and $\stWh{\eta}{\Lambda}{G, \circ}$,
respectively. 
Apparently, $K$ and $\lie{g}$ act on these spaces by the left
translation $\ell$. 
We call these modules the {\it standard Whittaker $\brgK$-modules} 
(actually, this definition is slightly different from that 
in \cite{Taniguchi 2013}). 
In \cite{Taniguchi 2013}, the author examined basic properties of
these modules. 
The main results of \cite{Taniguchi 2013} are 
(i) these modules are $K$-admissible and then have finite
length \cite[Corollary~2.4]{Taniguchi 2013}), 
(ii) if the infinitesimal character $\Lambda$ is {\it generic}, i.e. 
every irreducible $\brgK$-module which admits the infinitesimal
character $\Lambda$ is a principal series module, 
then 
$\WhL$ and $\WhLmg$ are
completely reducible, and 
(iii) the socle filtrations of these modules 
in the case when $G = U(n,1)$ and $\Lambda$ is non-singular integral 
are obtained.

The author investigated the socle filtrations of 
such modules in the case when $\Lambda$ is non-singular integral and 
$G = SL(3,\R)$ or $Sp(2,\R)$. 
The results are 
\eqref{eq:socle filtration of E(X(gamma_0)), PSU(2,1)}, 
\eqref{eq:socle filtration of E(X(gamma_0)), PSO(3,2)} and 
\eqref{eq:socle filtration of E(X(gamma_0')), PSO(4,1)} of this
paper. 
It should be noted that the method used at that time was direct
calculation of Whittaker functions, 
so it is quite different from the method used in this paper. 
The first observation of this investigation is that 
the global character of $\WhLmg$, i.e. the information of the
composition factors of it, is closely related to that of
the principal series module with infinitesimal character $\Lambda$. 
When $G$ is a connected real split semisimple Lie group with finite
center, 
the global character of $\WhL$ had been already obtained by 
Matumoto (\cite{M0}). 
From his result, we obtain the global character of the module
$\WhLmg$ for (not necessarily connected) general real split 
reductive linear Lie group. 
For simplicity, we abbreviate ``real split reductive
linear Lie group'' to ``split group'' from now on. 
Let $\ch V$ be the global character of a  $\brgK$-module $V$. 
Let $W$ be the little Weyl group $W(\lier{g}, \lier{a})$ and $|W|$ its
order. 
The result is as follows: 
\begin{theorem}{\rm (Matumoto \cite{M0} and 
Theorem~\ref{theorem:character of WhLmg})} 
Suppose $G$ is a split group. 
The global characters of $\WhL$ and $\WhLmg$ are given by 
\begin{align*}
& 
\ch \WhL 
= 
|W|\, \ch (\Ind_{AN}^{G}(e^{\Lambda+\rho} \otimes 1_{N})_{K}), 
&
&
\ch \WhLmg 
= 
\ch(\Ind_{AN}^{G}(e^{\Lambda+\rho} \otimes 1_{N})_{K}), 
\end{align*}
respectively. 
\end{theorem}

Another observation of the author's is that there seems to be 
a ``self-adjoint'' structure in the socle filtrations of $\WhL$ and
$\WhLmg$. 
The main purpose of this paper is to prove such self-adjoint
properties in the case when $G$ is a general split group. 
The result is as follows: 
\begin{theorem}
\label{theorem=7.4}
{\rm (Theorem~\ref{theorem:2nd duality theorem})} 
Suppose $G$ is a split group. 
Then 
\[
(\WhL)^{c} 
\simeq 
\stWh{\eta}{-\Lambda}{\circ}. 
\]
Here, $(\ast)^{c}$ denotes the contragredient $\brgK$-module of
$\ast$. 
If moreover $G=G_{\max}$, then 
\[
(\WhLmg)^{c} 
\simeq 
\stWh{\eta}{-\Lambda}{}. 
\]
\end{theorem}
For the definition of $G=G_{\max}$, see the proof of 
Theorem~\ref{theorem:character of WhLmg}. 
Note that, as is shown in \cite{Taniguchi 2013} for $G=U(n,1)$ 
and in \cite{Taniguchi 2015} for $G=Spin(r,1)$, 
if $G$ is not a split group, 
such self-adjoint properties do not hold. 

One of the important facts used in this research is 
that 
the modules $\WhL$ and $\WhLmg$ are injective objects in the category
$\mathcal{H}_{G}[\Lambda]^{(1)}$ 
of Harish-Chandra $\brgK$-modules which admit the infinitesimal
character $\Lambda$ 
(Corollary~\ref{corollary:injectivity of Wh modules}). 
This is a restatement of the well known fact that the functor of 
``taking Whittaker vectors'' is exact 
(Theorem~\ref{theorem:exactness of Wh vector}). 
From this injectivity and well known results on the Whittaker models, 
we obtain the direct sum decomposition of 
$\WhL$ and $\WhLmg$ into indecomposable injective modules. 
The results are stated as 
Corollary~\ref{corollary:WhL as injective module}. 
Also, this injectivity captures the behavior of the modules $\WhL$
and $\WhLmg$ under the translation functor 
$\psi_{\Lambda}^{\Lambda'}$. 
The results are stated as 
Theorem~\ref{theorem:behavior under the translation functors}. 

Initially the main purpose of this research was to prove the
abovementioned self-adjoint property. 
However, once such injective property is known, 
this research also has the significance of investigating the injective
objects in the category $\mathcal{H}_{G}[\Lambda]^{(1)}$. 
For a module $X \in \mathcal{H}_{G}[\Lambda]^{(1)}$, 
let $E(X)$ be its injective envelope 
(for the definition, 
see Definition~\ref{definition:injective envelope}) 
in $\mathcal{H}_{G}[\Lambda]^{(1)}$ 
and let $P(X)$ be the projective cover of $X$ in it. 
The main results on the injective envelope of an irreducible large
module are as follows: 
\begin{theorem}{\rm (Theorem~\ref{theorem:1st duality theorem} 
and Corollary~\ref{corollary:proj hull is inj env})}
Suppose $G$ is a split group. 
\begin{enumerate}
\item
If $X \in \mathcal{H}_{G}[\Lambda]^{(1)}$ is an irreducible large module, 
then its injective envelope $E(X)$ in $\mathcal{H}_{G}[\Lambda]^{(1)}$ 
has a unique irreducible quotient module $\widetilde{X}$, which is large. 
It satisfies 
\[
E(X)^{c} \simeq E(\widetilde{X}^{c}). 
\]
In other words, 
\[
E(X) = P(\widetilde{X}) 
\]
and 
$\WhL$, $\WhLmg$ are also projective modules in the category 
$\mathcal{H}_{G}[\Lambda]^{(1)}$. 
\item
The correspondence $X \mapsto \widetilde{X}$ is a permutation on
the set of irreducible large modules in
$\mathcal{H}_{G}[\Lambda]^{(1)}$. 
\item
Moreover, if $G = G_{\max}$, then 
$\widetilde{X} \simeq X$. 
\end{enumerate}
\end{theorem}

In Section~\ref{section:examples}, we present explicit examples of the
socle filtration of $E(X)$ when $G$ is a real rank two connected split
group and $X$ is an irreducible large module.

Before going ahead, we introduce notation used in this paper. 
For a real Lie group $L$, the Lie algebra of it is denoted by
$\lier{l}$ and its complexification by 
$\lie{l} = \lier{l} \otimes_{\R} \C$. 
Its universal enveloping algebra is denoted by $U(\lie{l})$. 
The center of $U(\lie{l})$ is denoted by $Z(\lie{l})$. 
This notation will be applied to groups denoted by other Roman letters
in the same way without comment. 
For a closed subgroup $L'$ of $L$ and a representation $\delta$ of $L'$, 
denote by $\Ind_{L'}^{L}(\delta)$ the representation of $L$ 
induced from $\delta$. 
For a compact Lie group $L$, the set of equivalence classes of
irreducible representations of $L$ is denoted by $\widehat{L}$.

Throughout this paper, 
let $G$ be a real reductive linear Lie group in the sense 
of \cite{Vogan Green}, 
and $G=KAN$ be its Iwasawa decomposition. 
We choose the set $\Sigma^{+}$ of positive roots in the root system
$\Sigma := \Sigma(\lier{g}, \lier{a})$ so that it corresponds to the
Lie algebra $\lier{n}$ of $N$. 
Half the sum of elements in $\Sigma^{+}$ with multiplicity is
denoted by $\rho$, as usual. 
The little Weyl group $W(\lier{g}, \lier{a})$ is denoted by $W$ and
its order by $|W|$. 
The longest element of $W$ with respect to the positive system
$\Sigma^{+}$ is denoted by $w_{0}$. 
The centralizer $Z_{G}(A)$ of $A$ in $G$ is denoted by $M$. 
Therefore, $MAN$ is a minimal parabolic subgroup of $G$. 

If $X$ is a representation of $G$, the subspace of $K$-finite vectors
in $X$ is denoted by $X_{K}$. 
For example, 
$\Ind_{MAN}^{G}(\sigma \otimes e^{\nu+\rho} \otimes 1_{N})_{K}$ is 
the Harish-Chandra $\brgK$-module of the principal series
representation induced from the tensor product representation of 
$\sigma \in \widehat{M}$, representation $e^{\nu + \rho}$ of $A$ 
($\nu \in \lie{a}^{\ast}$) and the trivial representation $1_{N}$ of
$N$. (We do not abbreviate the twist by $\rho$.)  
As in Theorem~\ref{theorem=7.4}, for a $\brgK$-module $V$, 
the $K$-finite contragredient $\brgK$-module of it is denoted by $V^{c}$. 

\noindent
\textbf{Acknowledgements.} 
The author would like to thank Professor Hiroshi Oda and 
Professor Noriyuki Abe who taught him the abovementioned 
interpretation of the exactness of the 
``taking the Whittaker vectors'' functor.


\section{Injectivity of the standard Whittaker $\brgK$-modules}
\label{section:injectivity}

In this section, we first recall some theorems of Whittaker models and
translate them into properties of $\WhL$ and $\WhLmg$. 

Let $\eta : N \rightarrow \C^{\times}$ be a unitary character of $N$. 
We denote the differential character of $\lier{n}$ to $\C$ 
and its complexification on $\lie{n}$ by the same letter $\eta$. 
We call $\eta$ {\it non-degenerate} if the differential of $\eta$ is 
non-trivial on the root space $(\lier{g})_{\alpha}$ for each simple
root $\alpha$ of $\Sigma(\lier{n}, \lier{a})$. 
Throughout this paper, we assume a unitary character of $N$ to be
non-degenerate. 
For a $\Ug$-module $V$, 
define {\it the space of algebraic Whittaker vectors 
$\Wh_{\eta}(V)$} by 
\[
\Wh_{\eta}(V) 
:=\{v \in V \mid X v = \eta(X) v, \ (X \in \lie{n})\}.
\]

We denote by $\mathcal{H}_{G}$ the category of Harish-Chandra $\brgK$-modules. 
The subcategory of $\mathcal{H}_{G}$ consisting of those Harish-Chandra 
$\brgK$-modules which admit generalized infinitesimal character $\Lambda$ 
(resp. infinitesimal character $\Lambda$) is denoted by 
$\mathcal{H}_{G}[\Lambda]$ (resp. $\mathcal{H}_{G}[\Lambda]^{(1)}$). 

For $V \in \mathcal{H}_{G}$, let $V_{\infty}$ be the
$C^{\infty}$-globalization of $V$. 
It is knows that this is a nuclear Fr\'{e}chet $G$-representation 
and $V \mapsto V_{\infty}$ is exact 
(\cite{Casselman}). 
We denote the algebraic dual $\Ug$-module of $V$ by $V^{\ast}$, 
and the topological dual $G$-representation 
of $V_{\infty}$ by $V_{\infty}'$. 
We define {\it the space of continuous Whittaker vectors 
$\Wh_{\eta}^{-\infty}(V)$} by 
\[
\Wh_{\eta}^{-\infty}(V) := 
\Wh_{\eta}(V_{\infty}') 
= 
\{w \in V_{\infty}' \mid X w = \eta(X) w, \ (X \in \lie{n})\}. 
\]
The correspondences $V \mapsto \Wh_{\eta}(V^{\ast})$ and 
$V \mapsto \Wh_{\eta}^{-\infty}(V)$ define functors from 
$\mathcal{H}_{G}$ to the category $\mathrm{Vect}_{f}$ of finite
dimensional vector spaces. 
\begin{theorem}
\label{theorem:exactness of Wh vector}
Suppose $\eta$ is a non-degenerate unitary character of $N$. 
\begin{enumerate}
\item {\rm (Kostant \cite{Kos}, Lynch \cite{Lynch})} 
The functor 
$\mathcal{H}_{G} \ni V 
\mapsto 
\Wh_{\eta}(V^{\ast}) \in \mathrm{Vect}_{f}$ 
is exact. 
\item {\rm (Casselman, cf. \cite{CHM})}
The functor 
$\mathcal{H}_{G} \ni V 
\mapsto \Wh_{\eta}^{-\infty}(V) \in \mathrm{Vect}_{f}$ 
is exact. 
\end{enumerate}
\end{theorem}

A vector $w \in \Wh_{-\eta}^{-\infty}(V)$ 
defines a continuous 
$G$-intertwining 
operator $\Phi$ from $V_{\infty}$ to $C^{\infty}(G/N; \eta)$ 
by the matrix coefficient map 
\[
V_{\infty} \ni v \mapsto \langle g \cdot w, v \rangle \in
C^{\infty}(G). 
\]
This map gives a linear isomorphism 
$\Wh_{-\eta}^{-\infty}(V) 
\simeq \Hom_{G}(V_{\infty}, C^{\infty}(G/N; \eta))$, 
where the right hand side is the space of continuous $G$-intertwining
operators. 
By a result of Wallach \cite{W}, 
for 
$V \in \mathcal{H}_{G}$, 
$\Phi \in \Hom_{\gK}(V, C^{\infty}(G/K; \eta)_{K})$ can be 
extended to an element of 
$\Hom_{G}(V_{\infty}, C^{\infty}(G/N; \eta))$ 
if and only if $\Phi(v)$ is a moderate growth function on
$G$ for every $v \in V$. 
Especially, if $V \in \mathcal{H}_{G}[\Lambda]^{(1)}$, 
then this is equivalent to $\Phi(V) \subset \WhLmg$. 
Therefore, Theorem~\ref{theorem:exactness of Wh vector}(2) means that 
$\mathcal{H}_{G}[\Lambda]^{(1)} \ni V \mapsto 
\Hom_{\gK}(V, \WhLmg) \in \mathrm{Vect}_{f}$ is exact. 

As for the space $\WhL$, every vector 
$v^{\ast} \in \Wh_{-\eta}(V^{\ast})$ 
defines a $\brgK$-intertwining operator 
$\Psi_{v^{\ast}} \in \Hom_{\gK}(V, C^{\infty}(G/N; \eta)_{K})$.
This is due to Goodman-Wallach (\cite{GW}) for quasi-split group $G$, 
and the proof for general $G$ is due to Matumoto (\cite{M1}).  
Therefore, Theorem~\ref{theorem:exactness of Wh vector}(1) implies
that 
$\mathcal{H}_{G}[\Lambda]^{(1)} \ni V \mapsto 
\Hom_{\gK}(V, \WhL) \in \mathrm{Vect}_{f}$ is exact, 
and then we can restate Theorem~\ref{theorem:exactness of Wh vector}
as follows: 
\begin{corollary}
\label{corollary:injectivity of Wh modules}
The modules $\WhL$ and $\WhLmg$ are injective modules 
in the category $\mathcal{H}_{G}[\Lambda]^{(1)}$. 
\end{corollary}


Next, we recall general theory of injective modules.
Let $R$ be an algebra and let $\mathcal{C}$ be a category of $R$-modules. 
We denote by $\Obj(\mathcal{C})$ and $\Mor(\mathcal{C})$ the objects and 
morphisms of $\mathcal{C}$, respectively. 
\begin{definition}
\label{definition:injective envelope}
Let $V \in \Obj(\mathcal{C})$. 
A pair $(E, i)$ ($E \in \Obj(\mathcal{C})$, $i \in \Mor(\mathcal{C})$) is 
an {\it injective envelope of $V$} in case $E$ is an injective object 
in $\mathcal{C}$ and 
\[
0 \rightarrow V \overset{i}{\rightarrow} E
\]
is an essential monomorphism. 
This is unique for $V$ up to isomorphism, and we denote it by 
$E(V)$. 
\end{definition}
%

\begin{definition}\label{definition:socle}
For $V \in \Obj(\mathcal{C})$, the {\it socle} of $V$ is 
the maximal semisimple submodule of $V$. 
We denote it by $\soc V$. 
\end{definition}
The following proposition is a consequence of the idempotent lifting
process (cf. {\cite[Part I, Proposition~3.16]{Jan}}). 
\begin{proposition}
\label{proposition:decomp of injective}
\begin{enumerate}
\item
An injective object $E$ of $\Obj(\mathcal{C})$ is indecomposable 
if and only if there exists a simple object $V$ of $\Obj(\mathcal{C})$ 
such that $E \simeq E(V)$. 
\item
An injective object $E \in \mathcal{C}$ is a direct sum of
indecomposable submodules. 
For each simple module $V \in \mathcal{C}$, 
the number of summands isomorphic to $E(V)$ is equal to 
the multiplicity of $V$ in $\soc E$. 
\end{enumerate}
\end{proposition}

The socles of $\WhL$ and $\WhLmg$ are known. 
\begin{theorem}[Kostant \cite{Kos}, Matumoto \cite{M1}]
\label{theorem:irred sub of WhL}
\begin{enumerate}
\item 
An irreducible module $V \in \mathcal{H}_{G}[\Lambda]$ is 
a submodule of $\WhL$ if and only if the Gelfand-Kirillov dimension 
$\Dim(V)$ of $V$ is equal to $\dim \lie{n}$. 
\item 
In this case, the multiplicity of $V$ in the socle 
of $\WhL$ is equal to the Bernstein degree $c_{\dim \lie{n}}(V)$ 
of $V$. 
\end{enumerate}
\end{theorem}
\begin{theorem}[Matumoto\cite{M2}]\label{theorem:irred sub of WhLmg}
\begin{enumerate}
\item
An irreducible module $V \in \mathcal{H}_{G}[\Lambda]$ 
is a submodule of $\WhLmg$ if and only if $\eta$ is contained in 
the wave front set $\mathrm{WF}(V)$ of $V$. 
\item
If $G$ is quasi-split and $V$ satisfies the condition in (1), 
then the multiplicity of $V$ in $\soc \WhLmg$ is one. 
\end{enumerate}
\end{theorem}
As corollaries of these theorems, we have the followings:
\begin{corollary}
\label{corollary:WhL as injective module}
Denote by $\mathcal{H}_{G}[\Lambda]_{\mathrm{irr}}$ the set of 
the equivalence classes of 
irreducible Harish-Chandra $\brgK$-modules with the infinitesimal 
character $\Lambda$. 
\begin{enumerate}
\item
$\displaystyle 
\WhL 
\simeq 
\bigoplus_{V \in \mathcal{H}_{G}[\Lambda]_{\mathrm{irr}}, 
\Dim(V) = \dim \lie{n}}
E(V)^{\oplus c_{\dim \lie{n}}(V)}$. 
\item
Let $m_{\eta}^{\infty}(V)$ be the multiplicity of $V$ 
in $\soc \WhLmg$. 
Then 
\[ 
\WhLmg 
\simeq 
\bigoplus_{V \in \mathcal{H}_{G}[\Lambda]_{\mathrm{irr}}, 
\eta \in \mathrm{WF}(V)}
E(V)^{\oplus m_{\eta}^{\infty}(V)}.
\]
\item
Especially, if $G$ is quasi-split, then 
\[ 
\WhLmg 
\simeq 
\bigoplus_{V \in \mathcal{H}_{G}[\Lambda]_{\mathrm{irr}}, 
\eta \in \mathrm{WF}(V)}
E(V).
\]
\end{enumerate}
\end{corollary}


\section{Behavior of $\WhL$ and $\WhLmg$ 
under the translation functors}
\label{section:translation of WhL}

Let $\psi_{\Lambda}^{\Lambda'}$ be the translation functor 
(cf. \cite[Chapter VII]{KV}) from 
$\mathcal{H}_{G}[\Lambda]$ to $\mathcal{H}_{G}[\Lambda']$. 
As an application of the injectivity of the modules $\WhL$ 
and $\WhLmg$, 
we will show 
Theorem~\ref{theorem:behavior under the translation functors} 
below. 

\begin{remark}
\label{remark:may assume dominant}
By definition, 
$\WhL$ and $\WhLmg$ satisfy 
$\WhL = \stWh{\eta}{w \Lambda}{\circ}$ and 
$\WhLmg = \stWh{\eta}{w \Lambda}{}$ for any $w \in W$. 
Therefore, when we discuss properties of them, 
in many situations we may restrict our consideration to $\Lambda$
which is contained in a closed Weyl chamber. 
\end{remark}
\begin{definition}
\label{definition:Weyl chamber}
\begin{enumerate}
\item
Let $\lie{h}$ be a Cartan subalgebra of $\lie{g}$. 
Let $\Delta = \Delta(\lie{g},\lie{h})$ be the root system and 
choose a positive system $\Delta^{+}$ of it. 
Define the corresponding positive Weyl chamber $C^{+}$ by 
\begin{align*}
C^{+} 
:= \{\nu \in \lie{h}^{\ast} 
\mid & \
\Real \langle \nu, \alpha \rangle > 0 
\ 
(\forall \alpha \in \Delta^{+})\}
\end{align*}
and the closure of it is denoted by $\overline{C^{+}}$. 
\item
If $G$ is a split group, 
$\lie{a}$ will be used as 
the Cartan subalgebra $\lie{h}$, 
and $\Sigma$, $\Sigma^{+}$ defined in
Section~\ref{section:introduction} are used as the root system and its
positive system, unless otherwise stated. 
\item
For two elements $\Lambda$ and $\Lambda'$ of $\overline{C^{+}}$, 
$\Lambda'$ is called {\it at least as singular as} $\Lambda$ if
$\langle \Lambda', \alpha \rangle = 0$ 
for every $\alpha \in \Sigma$ such that 
$\langle \Lambda, \alpha \rangle = 0$. 
\item
Denote by $\mathcal{L}$ the lattice of weights of finite
dimensional representations of $G$. 
We often regard an element of $\lie{a}^{\ast}$ as a one dimensional 
representation of $A$. 
\end{enumerate}
\end{definition}

\begin{theorem}
\label{theorem:behavior under the translation functors}
Let $G$ be a real reductive linear Lie group. 
Suppose two elements $\Lambda, \Lambda' \in \overline{C^{+}}$ 
satisfy $\Lambda' - \Lambda \in \mathcal{L}$ and 
assume that $\Lambda'$ is at least as singular as $\Lambda$. 
\begin{enumerate}
\item
Both 
$\psi_{\Lambda}^{\Lambda'} \WhL$ and 
$\psi_{\Lambda}^{\Lambda'} \WhLmg$ 
are injective objects in $\mathcal{H}_{G}[\Lambda']^{(1)}$. 
\item
Suppose moreover that $G$ is quasi-split. 
Then 
\begin{align*}
& 
\psi_{\Lambda}^{\Lambda'} \WhL
\simeq 
\stWh{\eta}{\Lambda'}{\circ}, 
&
& 
\psi_{\Lambda}^{\Lambda'} \WhLmg
\simeq 
\stWhmg{\eta}{\Lambda'}{}. 
\end{align*}
\end{enumerate}
\end{theorem}
\begin{proof}
In this proof, we will denote $\WhL$ or $\WhLmg$ by $\WhLast$. 

(1) Since $\Lambda'$ is at least as singular as
$\Lambda$ and 
since $\WhLast$ admits the infinitesimal character $\Lambda$, 
the proof of Thorems~7.171, 7.173 of \cite{KV} implies that 
 $\psi_{\Lambda}^{\Lambda'} \WhLast$ admits the infinitesimal
character $\Lambda'$. 
$\psi_{\Lambda}^{\Lambda'} \WhLast$ has finite length since 
$\WhLast$ does. 
It follows that $\psi_{\Lambda}^{\Lambda'} \WhLast$ is an object 
of $\mathcal{H}_{G}[\Lambda']^{(1)}$. 
The injectivity of $\psi_{\Lambda}^{\Lambda'} \WhLast$ follows from 
the injectivity of $\WhLast$, exactness of the translation functor and 
the adjoint property 
\[
\Hom_{\gK}(V, \psi_{\Lambda}^{\Lambda'} \WhLast)
\simeq 
\Hom_{\gK}(\psi_{\Lambda'}^{\Lambda}V, \WhLast). 
\]

(2) Since $\Lambda'$ is at least as singular as $\Lambda$, 
Theorem~7.229 of \cite{KV} states that 
$\psi_{\Lambda}^{\Lambda'}$ sends an irreducible module 
in $\mathcal{H}_{G}[\Lambda]$ to an irreducible module 
in $\mathcal{H}_{G}[\Lambda']$ or $0$. 
Moreover, by definition of ``large'', 
it sends a large module to a nonzero module. 
As stated in Theorem~\ref{theorem:irred sub of WhL}(2), 
for a large irreducible $\brgK$-module $V \in \mathcal{H}_{G}[\Lambda]$, 
the multiplicity of $V$ in $\soc \WhL$ is equal to  the Bernstein 
degree $c_{\dim \lie{n}}(V)$. 
Since $G$ is quasi-split, 
this is equal to $c_{\dim \lie{n}}(\psi_{\Lambda}^{\Lambda'}V)$ 
by Proposition~4.9 of \cite{Vogan GK-dim}. 
It follows that the socles of the two injective modules 
$\psi_{\Lambda}^{\Lambda'}\WhL$ and 
$\stWh{\eta}{\Lambda'}{\circ}$ are isomorphic. 
Therefore, they are isomorphic 
by Proposition~\ref{proposition:decomp of injective}. 

For the modules $\psi_{\Lambda}^{\Lambda'}\WhLmg$ and 
$\WhLmg$, 
the idea of proof is the same; compare their socles. 
If $V$ is an irreducible submodule of $\WhLmg$, then, 
by Theorem~\ref{theorem:irred sub of WhLmg}(2),  
the multiplicity of $V$ in the socle of $\WhLmg$ is one 
and $V$ satisfies $\eta \in \mathrm{WF}(V)$. 
But Lemma~5.2.6(2) of \cite{M2} implies that 
$\psi_{\Lambda}^{\Lambda'}V$ also satisfies 
$\eta \in \mathrm{WF}(\psi_{\Lambda}^{\Lambda'}V)$. 
On the other hand, $\psi_{\Lambda}^{\Lambda'}V$ is irreducible as
stated above. 
It follows that  $\psi_{\Lambda}^{\Lambda'}V$ 
is an irreducible submodule of $\stWh{\eta}{\Lambda'}{}$ 
and its multiplicity in $\soc \stWh{\eta}{\Lambda'}{}$ 
is one because of Theorem~\ref{theorem:irred sub of WhLmg}(2) again. 
From these, we have 
$\soc \psi_{\Lambda}^{\Lambda'} \WhLmg 
\simeq \soc \stWh{\eta}{\Lambda'}{}$, 
which implies 
$\psi_{\Lambda}^{\Lambda'} \WhLmg 
\simeq \stWh{\eta}{\Lambda'}{}$ 
by Proposition~\ref{proposition:decomp of injective}. 
\end{proof}


\section{The global characters of standard Whittaker $\brgK$-modules} 
\label{section:global characters}

In \cite{M0}, Matumoto examined the 
$\brgK$-module 
$\mathcal{A}_{K}(G/N, \eta; \mathcal{M}_{\Lambda})$, 
which is the same as our module $\WhL$. 
If $G$ is a split group, 
it has a filtration whose subquotients are isomorphic to 
principal series modules. 
More precisely, 
\begin{theorem}[{\cite[Corollary~9.2.4]{M0}}]
\label{theorem:Matumoto, comp series}
Suppose $G$ is a split group. 
Then there exist a numeration $w_{1}, w_{2}, \dots, w_{|W|}$ of the
elements of $W$ and $\brgK$-submodules $X_{0}, X_{1}, \dots, X_{|W|}$ of
$\WhL$ such that 
\begin{equation}\label{equation:std filtration of WhL} 
\WhL = X_{|W|} \supset X_{|W|-1} \supset \cdots 
\supset X_{1} \supset X_{0} = \{0\}
\end{equation}
and $X_{i}/X_{i-1}$ ($1 \leq i \leq |W|$) is isomorphic to 
$\Ind_{AN}^{G}(e^{w_{i}\Lambda+\rho} \otimes 1_{N})_{K}$ 
as a $\brgK$-module. 
If $\Lambda$ is taken from the positive closed Weyl chamber 
$\overline{C^{+}}$, then the identity element $e$ and the
longest element $w_{0}$ of $W$ can be taken as $w_{1}$ and $w_{|W|}$, 
respectively. 
\end{theorem}
\begin{remark}
\label{remark:generalization to our groups}
Though $G$ is assumed to be a real connected split semisimple Lie
group with finite center in \cite{M0}, 
it is easy to generalize its result 
for a general split group. 
In fact, suppose $G$ is a split group. 
Let $G_{0}$ be the connected component of $G$ containing
the identity element and $K_{0} := G_{0} \cap K$. 
Since $G_{0}$ is isomorphic to the product of the center of $G_{0}$
and the semisimple part of it, 
Theorem~\ref{theorem:Matumoto, comp series} for $G_{0}$ follows from 
Corollary~9.2.4 of \cite{M0}. 
Moreover, we will see in 
Section~\ref{section:realization, cohomological induction} 
that 
$\WhL$ is isomorphic to 
$\Gamma_{\lie{g}, \{e\}}^{\gK}
(\Hom_{\Zg \otimes \Un}(\Ug, \C_{\eta, \Lambda}))
\simeq 
\Gamma_{\lie{g},K_{0}}^{\gK}\stWh{\eta}{\Lambda}{G_{0},\circ}$. 
Here, $\Gamma$ is the Zuckerman functor. 
This is an analogue of the isomorphism 
$\Ind_{AN}^{G}(e^{\Lambda+\rho} \otimes 1_{N})_{K} 
\simeq 
\Gamma_{\lie{g}, \{e\}}^{\gK}
(\pro_{\lie{a}+\lie{n}}^{\lie{g}}\C_{\Lambda+\rho})
\simeq 
\Gamma_{\lie{g},K_{0}}^{\gK}\Ind_{AN}^{G_{0}}(e^{\Lambda+\rho} \otimes
1_{N})_{K}$ 
(cf. Section~\ref{section:realization, cohomological induction}). 
By these isomorphisms and the exactness of 
$\Gamma_{\lie{g},K_{0}}^{\gK}$, 
Theorem~\ref{theorem:Matumoto, comp series} for $G$ 
follows from the $G_{0}$ case of it. 
\end{remark}
\begin{corollary}\label{corollary:character of WhL}
Suppose $G$ is a split group. 
The global character of $\WhL$ is given by 
\begin{equation}\label{eq:global character of WhL} 
\ch \WhL 
= 
|W|\, \ch (\Ind_{AN}^{G}(e^{\Lambda+\rho} \otimes 1_{N})_{K}). 
\end{equation}
\end{corollary} 

In order to determine the global character of the module $\WhLmg$, 
we prepare one proposition on the structure of a submodule of
$\WhLmg$ constructed from a Jacquet integral. 
It will also be used in the calculation for the examples 
in Section~\ref{section:examples}. 

Let $P' = M'A'N'$ be a cuspidal parabolic subgroup of $G$ such that
$N'$ is contained in $N$ and $A'$ is contained in $A$. 
Let $\overline{N}'$ be the nilpotent subgroup opposite to $N'$. 
Denote the restriction of $\eta$ to $N'$ by $\eta_{P'}$ and to 
$M' \cap N$ by $\eta_{M'}$. 
Put 
$\rho' := (1/2)\tr (\ad(\lie{a}'))|_{\lie{n'}} \in (\lie{a}')^{\ast}$. 
For convenience, 
we realize the standard module induced from 
a discrete series $(\lie{m}', M' \cap K)$-module $(\sigma, V_{\sigma})$ 
and $\nu \in (\lie{a}')^{\ast}$ 
by 
$X(\sigma, \nu) 
:= 
\Ind_{M'A'\overline{N}'}^{G}
(\sigma \otimes e^{-\nu-\rho'} \otimes 1_{\overline{N}'})_{K}$, 
namely we realize it as a principal series module 
induced from the opposite parabolic $M'A'\overline{N}'$. 
We use the compact picture of this module, i.e. 
we identify the space $X(\sigma, \nu)$ with 
$\Ind_{M' \cap K}^{K}(\sigma|_{M' \cap K})_{K}$. 
Let $\lambda$ be an element of 
$\Wh_{-\eta_{M'}}^{-\infty}(V_{\sigma})$.
For the Jacquet integral, the following theorem is known 
%
\begin{theorem}
\label{theorem:holomorphy of Jacquet}
{\rm (Wallach \cite[Theorem~15.4.1]{Wallach real reductive II})}
Under the above situation, 
assume $\nu$ satisfies $\Real \langle \nu, \alpha \rangle > 0$ 
for every root $\alpha \in \Sigma(\lier{n}', \lier{a}')$. 
Then 
\[
J_{\sigma, \nu}(\lambda)(f)
:= 
\int_{N'} 
\lambda(f(n'))
\eta_{P'}(n')\, dn'
\]
converges absolutely for 
$f \in \Ind_{M' \cap K}^{K}(\sigma|_{M' \cap K})_{K}
\simeq 
X(\sigma, \nu)$ and 
the map $\nu \mapsto J_{\sigma, \nu}(\lambda)(f)$ is 
holomorphic. 
Moreover 
\[
J_{\sigma, \nu} : 
\Wh_{-\eta_{M'}}^{-\infty}(V_{\sigma}) 
\rightarrow 
\Wh_{-\eta}^{-\infty}(X(\sigma, \nu)) 
\]
is a linear bijection. 
\end{theorem}

Under the above situation, 
let us consider the ``analytic continuation" of the image of 
a Jacquet integral. 
We will construct a decreasing family of $\brgK$-submodules of
$X(\sigma, \nu)$ and a increasing family of $\brgK$-submodules of
$\WhLmg$, 
where $\Lambda$ is the infinitesimal character of $X(\sigma, \nu)$. 

Let $\epsilon$ be a sufficiently small positive constant. 
For $t \in D(\epsilon) := \{t \in \C \mid |t| < \epsilon\}$, 
consider the Jacquet integral $J_{\sigma, \nu + t \rho'}(\lambda)$ 
for $\lambda \in \Wh_{-\eta_{M'}}^{-\infty}(V_{\sigma})$, 
$\lambda \not= 0$. 
Since $J_{\sigma, \nu + t \rho'}(\lambda)$ is an element of 
$\Wh_{\eta}^{-\infty}(X(\sigma, \nu+t\rho'))$, 
\cite{W} implies that the function 
\[
J_{\sigma, \nu + t \rho'}(\lambda)(f)(x)
:=
J_{\sigma, \nu + t \rho'}(\lambda)(\ell(x^{-1})f) 
\quad 
(x \in G) 
\]
on $G$ 
for $f \in \Ind_{M' \cap K}^{K}(\sigma|_{M' \cap K})_{K}$ is contained
in $\stWh{\eta}{\Lambda+t\rho'}{}$, 
and it is holomorphic in $t$. 
Hereafter, we will abbreviate the variable $x$ and regard 
$J_{\sigma, \nu+t \rho'}(f)$ as an element of 
$\stWh{\eta}{\Lambda+t\rho'}{}$. 

For each $\delta \in \widehat{K}$, 
let $V(\delta)$ be the $\delta$-isotypic subspace of 
$\Ind_{M' \cap K}^{K}(\sigma|_{M' \cap K})_{K}$ 
and 
let $d(\delta)$ be its dimension. 
Choose a basis $\{f_{\delta,i}\}_{i=1}^{d(\delta)}$ of $V(\delta)$. 
We regard $f_{\delta,i}$ 
($\delta \in \widehat{K}$, $1 \leq i \leq d(\delta)$) 
as elements of 
$\Ind_{M' \cap K}^{K}(\sigma|_{M' \cap K})_{K}
\simeq 
X(\sigma, \nu+t\rho')$, $t \in D(\epsilon)$.  
Define 
\[
v_{\delta,i}(t) 
:= J_{\sigma, \nu+t\rho'}(\lambda)(f_{\delta,i}) 
\in \stWh{\eta}{\Lambda+t\rho'}{} 
\qquad
(i=1,2,\dots,d(\delta)). 
\]
These are holomorphic on $D(\epsilon)$. 
By taking $\epsilon$ small enough, 
we may assume 
$\langle \alpha \nu +t\rho', \check{\alpha} \rangle 
\not\in \Z$ 
for every non-imaginary root $\alpha$ when $t$ is contained in 
$D(\epsilon)^{\times} 
:= 
\{t \in \C \mid 0 < |t| < \epsilon\}$. 
In this case, $X(\sigma, \nu + t \rho')$ is irreducible 
by Theorem~1.1 of \cite{Speh-Vogan}. 
Since we are assuming $\lambda$ to be non-zero, 
if $t\not=0$, then $\{v_{\delta,i}(t)\}_{i=1}^{d(\delta)}$ are
linearly independent and they form a basis of the $\delta$-isotypic
subspace of the image of $J_{\sigma, \nu+t\rho'}(\lambda)$. 
By Proposition~2.21 of \cite{OS} and its proof, 
there exist non-negative integers $m_{\delta,i}$ 
and polynomial functions 
$c_{\delta,i,j}(t)$ ($i,j = 1, 2, \dots, d(\delta)$, $i \not= j$) 
such that 
\[
\widetilde{v}_{\delta,i}(t) 
= 
t^{-m_{\delta,i}} 
\left(
v_{\delta,i}(t) 
+ 
\sum_{j=1, \not=i}^{d(\delta)} 
c_{\delta,i,j}(t)\, v_{\delta,j}(t)
\right), 
\quad 
(i=1,2,\dots,d(\delta))
\]
are linearly independent if $t \in D(\epsilon)$ 
(even if $t=0$). 
Note that the non-negativity of $m_{\delta,i}$ comes from 
the vectors $v_{\delta,i}(t)$ ($1 \leq i \leq d(\delta)$) 
being holomorphic at $t=0$. 
Since $v_{\delta,1}(t), v_{\delta,2}(t), \dots, v_{\delta,d(\delta)}(t)$ 
are of moderate growth uniformly in $t \in D(\epsilon)$ 
(cf. \cite[\S14]{van den Ban}), 
$\tilde{v}_{\delta,1}(0), \tilde{v}_{\delta,2}(0), 
\dots, \tilde{v}_{\delta, d(\delta)}(0)$ are also
of moderate growth, 
and they admit the infinitesimal character $\Lambda$, 
namely they are contained in $\WhLmg$. 

Let $d(\delta)_{k}$ be the number of $\widetilde{v}_{\delta,i}(t)$ 
which satisfies $m_{\delta,i} \leq k$. 
By renumbering the subscript $i$, 
we arrange $\widetilde{v}_{\delta,i}(t)$ so that 
\begin{center}
if $d(\delta)_{k-1}+1 \leq i \leq d(\delta)_{k}$, 
then $m_{\delta,i} = k$. 
\end{center}
We put 
\[
\widetilde{f}_{\delta,i} 
:= 
f_{\delta,i} 
+ 
\sum_{j=1, \not=i}^{d(\delta)} 
c_{\delta,i,j}(t)\, f_{\delta,j} 
\in
\Ind_{M' \cap K}^{K}(\sigma|_{M' \cap K})_{K}, 
\quad 
(i=1,2,\dots,d(\delta)). 
\]
Then 
\[
\widetilde{v}_{\delta,i}(t) 
= t^{-m_{\delta,i}} J_{\sigma, \nu+t\rho'}(\lambda)
(\widetilde{f}_{\delta,i}). 
\]

Define subspaces $K^{(k)}$ ($k= 0,1,2, \dots$) 
of $\Ind_{M' \cap K}^{K}(\sigma|_{M' \cap K})_{K} 
\simeq 
X(\sigma, \nu)$ by 
\[
K^{(k)} 
:= 
\mathrm{Span}
\{\left. \widetilde{f}_{\delta,i} \right|_{t=0} 
\mid 
\delta \in \widehat{K}, \ 
i = d(\delta)_{k-1}+1, \dots, d(\delta)\}
\]
and subspaces $L^{(k)}$ ($k=0,1,2,\dots$) of $\WhLmg$ by 
\[
L^{(k)} 
:= 
\mathrm{Span}
\{\widetilde{v}_{\delta,i}(0) 
\mid 
\delta \in \widehat{K}, \ 
i=1,2,\dots,d(\delta)_{k}\}. 
\]
Then we have the following proposition: 
\begin{proposition}
\label{proposition:analy. conti of Jacquet} 
For each $k$, 
$K^{(k)}$ is a $\brgK$-submodule of $X(\sigma, \nu)$ and 
$L^{(k)}$ is a $\brgK$-submodule of $\WhLmg$. 
There is a $\brgK$-isomorphism 
\begin{equation}
\label{eq:analy. conti of Jacquet}
L^{(k)}/L^{(k-1)} 
\simeq 
K^{(k)}/K^{(k+1)}. 
\end{equation}
Moreover, the global character of 
$L^{\infty} := 
\cup_{k=0}^{\infty} 
L^{(k)}$ is the same as that of $X(\sigma, \nu)$. 
\end{proposition}
\begin{proof}
Consider the element $\widetilde{f}_{\delta,i}$ which satisfies 
$\delta \in \widehat{K}$ and 
$d(\delta)_{k-1}+1 \leq i \leq d(\delta)_{k}$. 
For $X \in \lie{g}$, we write 
\[
X \widetilde{f}_{\delta,i} 
= 
\sum_{\delta' \in \widehat{K}} 
\sum_{l \geq 0} 
\sum_{j=d(\delta')_{l-1}+1}^{d(\delta')_{l}} 
\alpha_{\delta,i}^{\delta', j}(X; t)\, 
\widetilde{f}_{\delta',j}. 
\]
By the way, we know that, 
if we choose a basis of $\Ind_{M' \cap K}^{K}(\sigma|_{M' \cap K})_{K}$ 
and express the action of $\lie{g}$ on the principal series module 
$X(\sigma, \nu+t \rho')$ 
by using this basis, 
the coefficients are polynomial functions with respect to 
the parameter $\nu + t \rho'$.  
Therefore, the coefficients $\alpha_{\delta,i}^{\delta', j}(X; t)$ 
are polynomials in $t$. 
Since 
\begin{align*}
X \widetilde{v}_{\delta,i}(t) 
&= 
t^{-k} J_{\sigma, \nu+t\rho'}(\lambda)(X \widetilde{f}_{\delta,i}) 
\\
&=
\sum_{\delta' \in \widehat{K}} 
\sum_{l \geq 0} 
\sum_{j=d(\delta')_{l-1}+1}^{d(\delta')_{l}} 
\alpha_{\delta,i}^{\delta', j}(X; t)\, 
t^{-k} 
J_{\sigma, \nu+t\rho'}(\lambda)(\widetilde{f}_{\delta',j}) 
\\
&= 
\sum_{\delta' \in \widehat{K}} 
\sum_{l \geq 0} 
\sum_{j=d(\delta')_{l-1}+1}^{d(\delta')_{l}} 
\alpha_{\delta,i}^{\delta', j}(X; t)\, 
t^{l-k}\, 
\widetilde{v}_{\delta',j}(t) 
\end{align*}
converges at $t=0$ and $\{\widetilde{v}_{\delta',j}(0)\}_{\delta', j}$
is linearly independent, 
$\lim_{t \to 0} \alpha_{\delta,i}^{\delta', j}(X; t) t^{l-k}$
converges even if $l<k$. 
Especially, $\alpha_{\delta,i}^{\delta', j}(X; 0) = 0$ if $j \leq
d(\delta)_{k-1}$. 
This implies that $L^{(k)}$ is $\Ug$-stable. 
Just in the same way, we can check that it is $K$-stable, 
so it is a $\brgK$-module. 
We also have 
\begin{equation}
\label{eq:Xv_delta,i-1}
X \widetilde{v}_{\delta,i}(0) 
\equiv 
\sum_{\delta' \in \widehat{K}} 
\sum_{j=d(\delta')_{k-1}+1}^{d(\delta')_{k}} 
\alpha_{\delta,i}^{\delta', j}(X; t)\, 
\widetilde{v}_{\delta',j}(0) 
\quad 
\mathrm{mod}\, L^{(k-1)}.  
\end{equation}
On the other hand, 
\begin{align}
\left. 
X \widetilde{f}_{\delta,i} 
\right|_{t=0} 
&= 
\sum_{\delta' \in \widehat{K}} 
\sum_{l \geq 0} 
\sum_{j=d(\delta')_{l-1}+1}^{d(\delta')_{l}} 
\alpha_{\delta,i}^{\delta', j}(X; t)\, 
\left. 
\widetilde{f}_{\delta',j}
\right|_{t=0} 
\notag\\
&=
\sum_{\delta' \in \widehat{K}} 
\sum_{l \geq k} 
\sum_{j=d(\delta')_{l-1}+1}^{d(\delta')_{l}} 
\alpha_{\delta,i}^{\delta', j}(X; 0)\, 
\left. 
\widetilde{f}_{\delta',j}
\right|_{t=0} 
\label{eq:Xv_delta,i-2}
\end{align}
since $\alpha_{\delta,i}^{\delta', j}(X; 0)=0$ 
if $j \leq d(\delta)_{k-1}$. 
This shows that $K^{(k)}$ is $\Ug$-stable. 
Th proof that it is $K$-stable is identical. 
We have proved that it is a $\brgK$-module. 
By comparing \eqref{eq:Xv_delta,i-1} and \eqref{eq:Xv_delta,i-2}, 
we obtain the isomorphism \eqref{eq:analy. conti of Jacquet}. 

Finally, by the construction of $\widetilde{f}_{\delta,i}$ and $K^{(k)}$, 
we have $K^{(0)}=X(\sigma, \nu)$ and $\cap_{k=0}^{\infty} K^{(k)} = 0$. 
Since the length of the $\brgK$-module $X(\sigma, \nu)$ is finite, 
\begin{align*}
\ch(L^{\infty}) 
= 
\sum_{k=0}^{\infty} 
\ch(L^{(k)}/L^{(k-1)}) 
\underset{\eqref{eq:analy. conti of Jacquet}}{=} 
\sum_{k=0}^{\infty} 
\ch(K^{(k)}/K^{(k+1)}) 
= 
\ch(K^{(0)}) 
= 
\ch(X(\sigma, \nu)). 
\end{align*}
\end{proof}

By using this proposition, we will show the following theorem.

\begin{theorem}\label{theorem:character of WhLmg} 
Suppose $G$ is a split group. 
For any infinitesimal character $\Lambda$, the grobal character of
$\WhLmg$ is given by 
\begin{equation}
\label{eq:character of WhLmg}
\ch \WhLmg 
= 
\ch(\Ind_{AN}^{G}(e^{\Lambda+\rho} \otimes 1_{N})_{K}). 
\end{equation}
\end{theorem}
\begin{proof}
We divide the proof into four steps. 
First, we define the group $G^{+}$. 

Let $G^{+}$ be a split group satisfying the
following three conditions:  
\begin{enumerate}
\item
$G \subset G^{+}$, 
\item
the Lie algebra of $G^{+}$ is $\lier{g}$, 
\item
$\Ad G^{+}$ contains the group 
$\{g \in \Ad \lie{g} \mid \Ad(g) \lier{g} \subset \lier{g}\}$. 
\end{enumerate}
Such group exists but is not unique. 
In \cite{M2}, it is called a type II envelope of $G$. 
Note that the number of principal nilpotent $G^{+}$-orbits in
$\lier{g}$ is one.  
We write $G=G_{\max}$ if $G$ itself satisfies the above three
conditions. 
We denote the maximal compact subgroup of $G^{+}$ containing $K$ 
by $K^{+}$.

\noindent
\textbf{Step I. $G = G_{\max}$.}

First, assume $G = G_{\max}$. 
In this case, Theorems~K and L of \cite{Kos} can be generalized for
the split group case 
and they imply 
\begin{equation}\label{equation:socles, in case G=Gmax}
\soc \stWh{\eta}{\Lambda}{G^{+}, \circ} 
\simeq 
(\soc \stWh{\eta}{\Lambda}{G^{+}})^{\oplus |W|}. 
\end{equation}
It follows from \eqref{equation:socles, in case G=Gmax}, 
Corollary~\ref{corollary:WhL as injective module} 
and \eqref{eq:global character of WhL} that 
\begin{equation}\label{equation:global character, G=Gmax}
\ch \stWh{\eta}{\Lambda}{G^{+}} 
= 
\frac{1}{|W|} \ch \stWh{\eta}{\Lambda}{G^{+}, \circ} 
= \ch (\Ind_{AN}^{G^{+}}(e^{\Lambda+\rho} \otimes 1_{N})_{K}). 
\end{equation}

\noindent
\textbf{Step II. 
$G$ is general and $\Lambda$ is generic.} 

Second, consider a general split group $G$. 
If $\Lambda$ is generic, i.e. every irreducible
$\brgK$-module in $\mathcal{H}_{G}[\Lambda]$ is a principal series, 
then Theorem~1.1 of \cite{Taniguchi 2013} implies the isomorphism 
\begin{equation}
\label{eq:WhLmg for generic Lambda}
\WhLmg 
\simeq 
\bigoplus_{\sigma \in \widehat{M}} 
\Ind_{MAN}^{G}(\sigma \otimes e^{\Lambda+\rho}\otimes 1_{N})_{K} 
\simeq 
\Ind_{AN}^{G}(e^{\Lambda+\rho}\otimes 1_{N})_{K}. 
\end{equation}
Therefore 
$\ch \WhLmg = \ch(\Ind_{AN}^{G}(e^{\Lambda+\rho}\otimes 1_{N})_{K})$. 

\noindent
\textbf{Step III. $G$ is general, $\Lambda$ is not generic and 
$\Lambda \in C^{+}$.} 

Third, consider the case when $G$ is general and $\Lambda$ is not 
generic. 
As is explained in Remark~\ref{remark:may assume dominant}, 
we may assume $\Lambda \in \overline{C^{+}}$. 
In this step, we will show \eqref{eq:character of WhLmg} in the case
when $\Lambda \in C^{+}$. 

Let $\{y_{1}, y_{2}, \dots, y_{l}\}$ be a set of representatives of 
$G \backslash G^{+}$. 
Since $M^{+} := Z_{G^{+}}(A)$ meets every component of
$G^{+}$, 
we may take $y_{1}, y_{2}, \dots, y_{l}$ from $M^{+}$. 
Thus we have 
\begin{align}
& 
\res_{\lie{g}, K^{+}}^{\gK} 
\stWh{\eta}{\Lambda}{G^{+}}
\simeq 
\bigoplus_{i=1}^{l} 
\stWh{y_{i} \cdot \eta}{\Lambda}{G}, 
\qquad 
f \mapsto (f_{1}, f_{2}, \dots, f_{l}), 
\label{equation:decomposition of WhL for Gmax to G}
\end{align}
where $f_{i}(g) = f(g y_{i})$ ($g \in G$, $i=1,2,\dots,l$) and 
\[
(y_{i} \cdot \eta)(n) = \eta({y_{i}}^{-1} n y_{i}) \quad (n \in N). 
\]
Note that the infinitesimal character $\chi_{\Lambda}$ 
is fixed by $G^{+}$, since $G^{+}$ is inner. 
In the same way as 
\eqref{equation:decomposition of WhL for Gmax to G}, we have 
\begin{equation}\label{equation:decomposition of PS for Gmax to G}
\res_{\lie{g}, K^{+}}^{\gK} 
\Ind_{AN}^{G^{+}}(e^{\Lambda+\rho} \otimes 1_{N})_{K^{+}}
\simeq 
\bigoplus_{i=1}^{l} 
\Ind_{AN}^{G}(e^{\Lambda+\rho} \otimes 1_{N})_{K}. 
\end{equation}
By \eqref{equation:decomposition of WhL for Gmax to G}, 
\eqref{equation:global character, G=Gmax} and 
\eqref{equation:decomposition of PS for Gmax to G}, 
we obtain the identity 
\begin{equation}\label{equation:1st character corresp} 
\begin{split}
\sum_{i=1}^{l} 
\ch \stWh{y_{i} \cdot \eta}{\Lambda}{G} 
&\underset{\eqref{equation:decomposition of WhL for Gmax to G}}{=} 
\ch \stWh{\eta}{\Lambda}{G^{+}} 
\underset{\eqref{equation:global character, G=Gmax}}{=}
\ch(\Ind_{AN}^{G^{+}}(e^{\Lambda+\rho} \otimes 1_{N})_{K^{+}})
\\
&\underset{\eqref{equation:decomposition of PS for Gmax to G}}{=}
l \times \ch(\Ind_{AN}^{G}(e^{\Lambda+\rho} \otimes 1_{N})_{K}). 
\end{split}
\end{equation}

Now, we apply Proposition~\ref{proposition:analy. conti of Jacquet} 
to the original Jacquet integral (\cite{Jac})
\begin{equation*}
\begin{split}
& 
J_{-w_{0}\Lambda} : 
\Ind_{A\overline{N}}^{G}
(e^{w_{0}\Lambda-\rho}\otimes 1_{\overline{N}})_{K} 
\simeq 
\bigoplus_{\sigma \in \widehat{M}} 
\Ind_{MA\overline{N}}^{G}
(\sigma \otimes e^{w_{0}\Lambda-\rho}\otimes 1_{\overline{N}})_{K} 
\rightarrow 
\stWh{\eta'}{\Lambda}{G}, 
\\ 
& 
J_{-w_{0}\Lambda}(f)(x) 
= 
\int_{N} f(xn) \eta'(n) dn, 
\quad
(x \in G) 
\end{split}
\end{equation*} 
for $\eta'=y_{i} \cdot \eta$. 
If $\Lambda$ is contained in $C^{+}$, 
Proposition~\ref{proposition:analy. conti of Jacquet} implies that 
there are (non-virtual) characters $\Theta_{i}$ 
($i=1,2,\dots,l$) of $\brgK$-modules in $\mathcal{H}_{G}[\Lambda]$ such
that 
\[
\ch \stWh{y_{i} \cdot \eta}{\Lambda}{G} 
= 
\ch(\Ind_{A\overline{N}}^{G}
(e^{w_{0}\Lambda-\rho} \otimes 1_{\overline{N}})_{K}) + \Theta_{i}
= 
\ch(\Ind_{AN}^{G}
(e^{\Lambda+\rho} \otimes 1_{N})_{K}) + \Theta_{i}. 
\]
Therefore, by \eqref{equation:1st character corresp}, 
\[
l \times 
\ch(\Ind_{AN}^{G}(e^{\Lambda+\rho} \otimes 1_{N})_{K}) 
= 
\sum_{i=1}^{l} 
\ch \stWh{y_{i} \cdot \eta}{\Lambda}{G} 
= 
l \times 
\ch(\Ind_{AN}^{G}(e^{\Lambda+\rho} \otimes 1_{N})_{K}) 
+ 
\sum_{i=1}^{l} \Theta_{i}, 
\]
so $\Theta_{i} = 0$ for $i=1,2,\dots,l$. 
It follows that 
$\ch \stWh{y_{i} \cdot \eta}{\Lambda}{G} 
= 
\ch(\Ind_{AN}(e^{\Lambda+\rho} \otimes 1_{N})_{K})$. 

\noindent
\textbf{Step IV. $G$ is general, $\Lambda$ is not generic and 
$\Lambda \in \overline{C^{+}}$.}

In this case, choose a non-singular infinitesimal character, say 
$\Lambda+2\rho \in C^{+}$, 
and use the translation functor 
$\psi_{\Lambda+2\rho}^{\Lambda}$. 
By Theorem~\ref{theorem:behavior under the translation functors}, 
$\psi_{\Lambda+2\rho}^{\Lambda}\stWh{\eta}{\Lambda+2\rho}{} 
\simeq \WhLmg$, and by Step III, 
$\ch(\stWh{\eta}{\Lambda+2\rho}{}) 
= \ch(\Ind_{AN}^{G}(e^{\Lambda+2\rho+\rho} \otimes 1_{N})_{K})$. 
Since 
$\psi_{\Lambda+2\rho}^{\Lambda}$ is a translation functor 
from non-singular to singular, 
we have 
$\psi_{\Lambda+2\rho}^{\Lambda}
\Ind_{AN}^{G}(e^{\Lambda+2\rho+\rho} \otimes 1_{N})_{K}
\simeq \Ind_{AN}^{G}(e^{\Lambda+\rho} \otimes 1_{N})_{K}$. 
Then the theorem also holds for this case. 
\end{proof}


\section{Realization of standard Whittaker $\brgK$-modules by 
cohomological induction}
\label{section:realization, cohomological induction}

As stated in 
Remark~\ref{remark:generalization to our groups}, 
the following isomorphism is known: 
\[
\Ind_{AN}^{G}(e^{\Lambda+\rho} \otimes 1_{N})_{K} 
\simeq 
\Gamma_{\lie{g}, \{e\}}^{\gK}
(\pro_{\lie{a}+\lie{n}}^{\lie{g}}\C_{\Lambda+\rho})
\]
(cf. \cite[Proposition~6.3.5]{Vogan Green} or 
\cite[Propostion~11.47]{KV}). 
Here, $\Gamma$ is the Zuckerman functor and 
$\pro_{\lie{a}+\lie{n}}^{\lie{g}}\C_{\Lambda+\rho}
= 
\Hom_{U(\lie{a}+\lie{n})}(\Ug, \C_{\Lambda+\rho})$. 
We can realize our module $\WhL$ in the same way. 
One of the benefits of this realization is that we can treat the
contragredient module of $\WhL$ by algebraic methods. 

Let $\C_{\Lambda,\eta}$ be the one dimensional 
$\Zg \otimes \Un$ module with basis $1_{\Lambda,\eta}$ defined
by 
\[
(z \otimes u) 1_{\Lambda,\eta} 
= \chi_{\Lambda}(z)\, \eta(u) 1_{\lambda, \eta}, 
\qquad 
(z \in \Zg, u \in \Un). 
\]
\begin{proposition}\label{proposition:realization of WhL by coh-ind} 
Let $G$ be a split group. 
\begin{enumerate}
\item
There is a $\brgK$-module isomorphism 
\begin{equation}
\label{equation:realization of WhL by coh-ind, 1} 
\WhL 
\simeq 
\Gamma_{\lie{g}, \{e\}}^{\gK} 
(\Hom_{\Zg \otimes \Un}(\Ug, \C_{\Lambda, \eta})). 
\end{equation}
Here, $\Zg \otimes \Un$ and $\Ug$ act on $\Ug$ by the left and right
multiplication, respectively. 
\item
For any $i > 0$, 
\[
(\Gamma_{\lie{g}, \{e\}}^{\gK})^{i} 
(\Hom_{\Zg \otimes \Un}(\Ug, \C_{\Lambda, \eta}))
= 0. 
\]
Here, $(\Gamma_{\lie{g}, \{e\}}^{\gK})^{i}$ is the $i$-th derived
functor of ${\Gamma_{\lie{g}, \{e\}}^{\gK}}$. 
\end{enumerate}
\end{proposition}
\begin{proof}
The proof of this proposition is almost the same as that of 
\cite[Proposition~6.3.5]{Vogan Green} or 
\cite[Propostion~11.47]{KV}. 
For completeness, we write it here. 

(1) First, we show that, for each $K$-type $\delta \in \widehat{K}$, 
the multiplicities of $\delta$ in both sides of 
\eqref{equation:realization of WhL by coh-ind, 1} are the same. 
By Theorem~\ref{theorem:Matumoto, comp series}, 
$\res_{\gK}^{\lie{k},K} \WhL 
\simeq C(K)^{\oplus |W|}$, where $C(K) = C^{\infty}(K)_{K}$ is the
space of $K$-finite smooth functions on $K$.  
Therefore, the multiplicity of $\delta$ in $\WhL$ is 
$|W| \dim \delta$. 

In order to count the multiplicity in the right hand side of 
\eqref{equation:realization of WhL by coh-ind, 1}, 
we use the following well known fact: 
\begin{proposition}[cf. {\cite[\S~3.7]{Wallach real reductive I}}]
\label{proposition:Osborne?}
Suppose $G$ is a split group. 
There is a $|W|$-dimensional
linear subspace $E$ of $U(\lie{g})$ such that 
\[\Ug \simeq U(\lie{k}) \otimes E \otimes \Zg \otimes \Un
\simeq \Un \otimes \Zg \otimes E \otimes U(\lie{k}). 
\]
\end{proposition}

By this proposition, 
\begin{align*}
\Hom_{K}&(\delta, 
\Gamma_{\lie{g}, \{e\}}^{\gK} 
(\Hom_{\Zg \otimes \Un}(\Ug, \C_{\Lambda, \eta})))
\\
& \simeq 
\Hom_{\lie{k}}(\res_{\lie{k},K}^{\lie{k},\{e\}} \delta, 
\Hom_{\Zg \otimes \Un}(\Ug, \C_{\Lambda, \eta}))
\\
& \simeq 
\Hom_{\lie{k}}(\res_{\lie{k},K}^{\lie{k},\{e\}} \delta, 
\Hom_{\Zg \otimes \Un}(\Un \Zg E U(\lie{k}), \C_{\Lambda, \eta}))
\\
& \simeq 
\Hom_{\C}(\delta, E). 
\end{align*}
Therefore, the multiplicity of $\delta$ in the right
hand side of \eqref{equation:realization of WhL by coh-ind, 1} is 
$\dim E \times \dim \delta 
= 
|W| \dim \delta$, which is the same as the multiplicity of $\delta$ in
$\WhL$. 

Second, we show that there is an injective $\brgK$-homomorphism from 
$\WhL$ to 
$\Gamma_{\lie{g}, \{e\}}^{\gK} 
(\Hom_{\Zg \otimes \Un}(\Ug, \C_{\Lambda, \eta}))$. 
As in \cite{KV}, we denote by $R(\gK)$ the Hecke algebra of the pair
$\brgK$. 
There are natural isomorphisms 
(the isomorphism \eqref{equation:pf of 5.1-1} is due to
\cite[Theorem~2.69]{KV})
\begin{align}
& 
\Hom_{\gK}(\WhL, 
\Gamma_{\lie{g}, \{e\}}^{\gK} 
(\Hom_{\Zg \otimes \Un}(\Ug, \C_{\Lambda, \eta})))
\notag
\\
& \simeq 
\Hom_{R(\gK)}(\WhL, 
\Hom_{R(\lie{g}, \{e\})}(R(\gK), 
\Hom_{\Zg \otimes \Un}(\Ug, \C_{\Lambda, \eta}))_{K})
\label{equation:pf of 5.1-1}
\\
& \simeq 
\Hom_{R(\lie{g}, \{e\})}(\WhL, 
\Hom_{\Zg \otimes \Un}(\Ug, \C_{\Lambda, \eta}))
\notag
\\
& \simeq 
\Hom_{\Ug}(\WhL, 
\Hom_{\Zg \otimes \Un}(\Ug, \C_{\Lambda, \eta})). 
\label{equation:pf of 5.1-2}
\end{align}
Define an element $\phi_{2}$ in \eqref{equation:pf of 5.1-2} by 
\begin{align*}
\phi_{2} : &\WhL \rightarrow 
\Hom_{\Zg \otimes \Un}(\Ug, \C_{\Lambda, \eta}), 
\\ 
&\phi_{2}(f)(u) = (\ell(u)f)(e) \quad (f \in \WhL, u \in \Ug). 
\end{align*}
$\phi_{2}$ is well defined. 
In fact, for any $u_{1} \in \Un$
$z \in \Zg$ and $u \in \Ug$, 
\begin{align*}
\phi_{2}(f)(u_{1} z u) 
&= 
(\ell(u_{1}) \ell(z) \ell(u) f)(e) 
= (r({}^{t}u_{1}) \ell(z) \ell(u)f)(e) 
\\
&= \eta(u_{1}) \chi_{\Lambda}(z) (\ell(u)f)(e) 
= \eta(u_{1}) \chi_{\Lambda}(z) \phi_{2}(f)(u),
\end{align*}
where $r$ is the right translation, 
and for any $u, u' \in \Ug$ 
\[
\phi_{2}(\ell(u')f)(u) = (\ell(u) \ell(u')f)(e) 
= (\ell(uu')f)(e) = \phi_{2}(f)(uu'). 
\]
Let $\phi_{1}$ be 
the element in \eqref{equation:pf of 5.1-1} 
which corresponds to $\phi_{2}$ by 
\[
\phi_{1}(f)(a)(u) = \phi_{2}(\ell(a)f)(u) = (\ell(u) \ell(a) f)(e), 
\quad a \in R(\gK). 
\]
Suppose $\phi_{1}(f) = 0$ holds for $f \in \WhL$. 
Then 
$\phi_{1}(f)(a)(u) = \phi_{2}(\ell(a)f)(u) 
= (\ell(u) \ell(a)f)(e) = 0$ 
for any $a \in R(\gK)$ 
and $u \in \Ug$, 
especially 
\begin{equation}\label{equation:condition for phi2(f)=0} 
(\ell(u) \ell(k)f)(e)=0 
\quad 
\mbox{for any }  k \in K \mbox{ and } u \in \Ug.
\end{equation}
$\ell(k)f$ is a real analytic function on $G$ since it is $K$-finite. 
Therefore $\ell(k)f = 0$ 
by \eqref{equation:condition for phi2(f)=0}. 
Since $K$ meets every connected component of $G$, 
$f=0$. 
It follows that $\phi_{1}$ is injective. 
This completes the proof of 
\eqref{equation:realization of WhL by coh-ind, 1}. 

The proof of 
Proposition~\ref{proposition:realization of WhL by coh-ind}(2) 
is almost the same as that of Proposition~6.3.5 of 
\cite{Vogan Green}. 
Let $\mathcal{F}$ denote the forgetful functor. 
Since 
$\mathcal{F}_{\gK}^{\kK} \circ \Gamma_{\lie{g},\{e\}}^{\gK} 
\simeq 
\Gamma_{\lie{k}, \{e\}}^{\kK} 
\circ \mathcal{F}_{\lie{g},\{e\}}^{\lie{k},\{e\}}$ 
(\cite[Proposition~2.69]{KV}),
\begin{align*}
\res_{\gK}^{\kK} 
&
(\Gamma_{\lie{g},\{e\}}^{\gK})^{i}
(\Hom_{\Zg \otimes \Un}(\Ug, \C_{\Lambda,\eta}))
\\
& \simeq 
(\Gamma_{\lie{k},\{e\}}^{\kK})^{i}
\mathcal{F}_{\lie{g},\{e\}}^{\lie{k},\{e\}} 
(\Hom_{\Zg \otimes \Un}(\Ug, \C_{\Lambda,\eta}))
\\
& \simeq 
(\Gamma_{\lie{k},\{e\}}^{\kK})^{i}
(\Hom_{\C}(E \otimes_{\C} \Uk, \C)) 
\qquad 
\mbox{(by Proposition~\ref{proposition:Osborne?})}
\\
& \simeq 
(\Gamma_{\lie{k},\{e\}}^{\kK})^{i}
(\Hom_{\C}(\Uk, \C^{\dim E})) 
\\
& = 
(\Gamma_{\lie{k},\{e\}}^{\kK})^{i}
(\pro_{0,\{e\}}^{\lie{k},\{e\}}(\C^{\dim E})).
\end{align*}
But since $\pro_{0,\{e\}}^{\lie{k},\{e\}}(\C^{\dim E})$ 
is an injective $(\lie{k},\{e\})$-module 
(\cite[Corollary~6.1.24]{Vogan Green}), 
this module is zero if $i>0$. 
This proves (2). 
\end{proof}

We will realize the contragredient module of $\WhL$ algebraically. 
Define 
\[
\indWhL 
:= 
\Ug \otimes_{\Zg \otimes \Un} \C_{\Lambda,\eta}. 
\]
The contragredient module ${\indWhL}^{c}$ of it is 
\begin{align*}
{\indWhL}^{c}
&= 
\Hom_{\C}(\indWhL, \C) 
\\
&\simeq
\Hom_{\C}(\C_{\Lambda,\eta} \otimes_{\Zg \otimes \Un} \Ug, \C) 
\\
& \simeq 
\Hom_{\Zg \otimes \Un}(\Ug, \Hom_{\C}(\C_{\Lambda,\eta}, \C)) 
\\
& \simeq 
\Hom_{\Zg \otimes \Un}(\Ug, \C_{-\Lambda,-\eta}).
\end{align*}

Recall that the Bernstein functor $\Pi_{\lie{g}, \{e\}}^{\gK}$ 
is defined by 
\[
\Pi_{\lie{g}, \{e\}}^{\gK} V 
= R(\gK) \otimes_{R(\lie{g}, \{e\})} V 
\]
for a $\Ug$-module $V$ (cf \cite{KV}). 
By the above discussion, 
we obtain the following proposition from 
Proposition~\ref{proposition:realization of WhL by coh-ind} 
and \cite[Theorem~3.1]{KV}. 
\begin{proposition}
\label{proposition:realization of dual Whittaker gK module}
Suppose $G$ is a split group. 
\begin{enumerate}
\item
The contragredient $\brgK$-module of $\WhL$ is isomorphic 
to 
$\duWh{-\Lambda}{-\eta}$. 
\item
If $i>0$, then 
$(\Pi_{\lie{g},\{e\}}^{\gK})_{i}(\indWh{-\Lambda}{-\eta}) = 0$. 
\end{enumerate}
\end{proposition}


\section{Behavior of $\duWhL$ under the translation functors}
\label{section:translation of duWhL}

In this section, we investigate the behavior of $\duWhL$ under
the translation functor $\psi_{\Lambda}^{\Lambda'}$. 
We maintain the assumption that $G$ is a split group.

For two elements $\Lambda, \Lambda' \in \lie{a}^{\ast}$, 
suppose that the difference $\Lambda' - \Lambda$ is contained in 
$\mathcal{L}$. 
Since $\indWhL$ has the infinitesimal character $\Lambda$, 
$\duWhL$ also does so by \cite[Theorem~5.21(b)]{KV}. 
Then by the proof of Theorem~7.237 of \cite{KV}, we have 
\begin{equation}\label{equation:psi and Pi}
\psi_{\Lambda}^{\Lambda'}(\duWhL) 
\simeq 
\Pi_{\lie{g},\{e\}}^{\gK}(\psi_{\Lambda}^{\Lambda'}(\indWhL)). 
\end{equation}
\begin{lemma}\label{lemma:translation of duWhL to wall}
Let $G$ be a split group.  
Suppose $\Lambda, \Lambda' \in \overline{C^{+}}$ satisfy 
$\Lambda' - \Lambda \in \mathcal{L}$. 
If $\Lambda'$ is at least as singular as $\Lambda$, 
then 
\[
\psi_{\Lambda}^{\Lambda'}(\duWhL) 
\simeq 
\duWh{\Lambda'}{\eta}. 
\]
\end{lemma}
\begin{proof}
By Theorem~3.6.1 of \cite{Kos}, 
$\indWhL$ is irreducible. 
$\psi_{\Lambda}^{\Lambda'}(\indWhL)$ is non-zero by Theorem~4.6 there. 
Since $\Lambda'$ is at least as singular as $\Lambda$, 
$\psi_{\Lambda}^{\Lambda'}(\indWhL)$ is irreducible by Theorem~7.229
of \cite{KV}. 
Therefore, $\psi_{\Lambda}^{\Lambda'}(\indWhL)$ is isomorphic to 
$\indWh{\Lambda'}{\eta}$ again by Theorem~3.6.1 of \cite{Kos}. 
The lemma follows from \eqref{equation:psi and Pi}. 
\end{proof}

In order to analyze the translation to converse direction, i.e. 
singular to non-singular, we need some preparation. 

\begin{lemma}\label{lemma:multi of Langlands quotients of PS}
Let $G$ be a split group. 
Suppose $\Lambda \in C^{+}$. 
For $\sigma \in \widehat{M}$, let $V$ be the Langlands quotient of 
the principal series module 
$\Ind_{MAN}^{G}
(\sigma \otimes e^{\Lambda + \rho} \otimes 1_{N})_{K}$ 
induced from the minimal parabolic subgroup $MAN$. 
Then there is a unique irreducible large module 
$X \in \mathcal{H}_{G}[\Lambda]$ which satisfies 
\begin{enumerate}
\item
its injective envelope $E(X)$ is a direct summand of $\WhLmg$, 
and 
\item
$V$ is a composition factor of $E(X)$. 
\end{enumerate}
In this case, the multiplicity of $V$ in $E(X)$ is one. 
\end{lemma}

\begin{proof}
We use the decomposition 
\begin{equation}
\label{equation:isom between Ind}
\Ind_{AN}^{G}
(e^{\Lambda+\rho} \otimes 1_{N})_{K}
\simeq 
\bigoplus_{\sigma \in \widehat{M}}
\Ind_{MAN}^{G}
(\sigma \otimes e^{\Lambda+\rho}
\otimes 1_{N})_{K}.
\end{equation}
By this, Theorem~\ref{theorem:character of WhLmg} 
and the Langlands classification, 
the multiplicity of $V$ in $\WhLmg$ is one. 
We know from Corollary~\ref{corollary:WhL as injective module}(3) 
that $\WhLmg$ is a direct sum of
injective envelopes of irreducible large modules, each of which
appears once. 
Therefore, there is just one large module $X$ which satisfies 
the conditions (1) and (2). 
\end{proof}
\begin{corollary}
\label{corollary:multi of E(Xj) in WhL}
Under the assumption of 
Lemma~\ref{lemma:multi of Langlands quotients of PS}, 
let $\{X_{1}, X_{2}, \dots, X_{l}\}$ be the set of large irreducible
modules in $\mathcal{H}_{G}[\Lambda]$ such that 
each injective envelope $E(X_{j})$ ($j=1,2,\dots,l$) contains $V$ as a
composition factor. 
Then 
\begin{equation}
\label{eq:multi of V in WhL}
\sum_{j=1}^{l} c_{\dim \lie{n}}(X_{j}) = |W|. 
\end{equation}
Here, $c_{\dim \lie{n}}(X_{j})$ is the Bernstein degree of $X_{j}$, 
which is equal to the multiplicity of $E(X_{j})$ in 
the decomposition 
Corollary~\ref{corollary:WhL as injective module}(1) 
of $\WhL$. 
\end{corollary}
\begin{proof}
By Theorem~\ref{theorem:irred sub of WhLmg}, 
for each $X_{j}$, there is a non-degenerate character
$\eta'$ of $N$ such that $X_{j}$ is an irreducible submodule of 
$\stWh{\eta'}{\Lambda}{}$, 
and therefore $E(X_{j})$ is a direct summand 
of $\stWh{\eta'}{\Lambda}{}$. 
Since $E(X_{j})$ contains $V$ as a composition factor, 
the conditions of Lemma~\ref{lemma:multi of Langlands quotients of PS}
are met, so the multiplicity of $V$ in $E(X_{j})$ is one. 
This fact and the decomposition in 
Corollary~\ref{corollary:WhL as injective module}(1) imply 
that the left had side of\eqref{eq:multi of V in WhL} is the
multiplicity of $\ch(V)$ in 
\[
\ch(\WhL)
=|W| \ch(\Ind_{AN}^{G}(e^{\Lambda+\rho} \otimes 1_{N})_{K}) 
= 
|W| 
\sum_{\sigma \in \widehat{M}} 
\ch(
\Ind_{MAN}^{G}(\sigma \otimes e^{\Lambda+\rho} \otimes 1_{N})_{K}).  
\]
The first equality is due to Corollary~\ref{corollary:character of WhL}. 
Since $V$ is the Langlands quotient of a principal
series induced from the minimal parabolic subgroup, 
there is just one 
$\sigma \in \widehat{M}$ such that $V$ is a composition factor of 
$\Ind_{MAN}^{G}(\sigma \otimes e^{\Lambda+\rho} \otimes 1_{N})_{K}$. 
Therefore, the multiplicity of $\ch(V)$ in $\ch(\WhL)$ is $|W|$. 
\end{proof}

\begin{lemma}\label{lemma:genPS to Pi0-2}
Suppose $G$ is a split group.
Then for any $\sigma \in \widehat{M}$ and $\Lambda \in C^{+}$, 
\begin{equation}
\label{equation:genPS to Pi0-2}
\dim 
\Hom_{\gK}
(\Ind_{MAN}^{G}(\sigma \otimes e^{\Lambda+\rho} \otimes 1_{N})_{K},  
\duWhL) 
=
|W|. 
\end{equation}
\end{lemma}
\begin{proof}
We consider the space of dual homomorphisms 
\begin{align}
\Hom_{\gK}&
(\Ind_{MAN}^{G}(\sigma \otimes e^{\Lambda+\rho} \otimes 1_{N})_{K},  
\duWhL) 
\notag
\\
&\simeq 
\Hom_{\gK}(\stWh{-\eta}{-\Lambda}{\circ}, 
\Ind_{MAN}^{G}
(\sigma^{c} \otimes e^{-\Lambda+\rho} \otimes 1_{N})_{K})
\label{equation:Hom WhL to PS}
\end{align}
given by 
Proposition
~\ref{proposition:realization of dual Whittaker gK module}(1). 
Here, $\sigma^{c}$ is the contragredient representation of $\sigma$. 

First, we show that the dimension 
of \eqref{equation:Hom WhL to PS} is at most $|W|$. 
In fact, 
let $V$ be the Langlands submodule of 
$\Ind_{MAN}^{G}
(\sigma^{c} \otimes e^{-\Lambda+\rho} \otimes 1_{N})_{K}$. 
This is the unique irreducible submodule of 
$\Ind_{MAN}^{G}
(\sigma^{c} \otimes e^{-\Lambda+\rho} \otimes 1_{N})_{K}$ 
by the Langlands classification 
and the condition $\Lambda \in C^{+}$. 
Therefore, 
the dimension of \eqref{equation:Hom WhL to PS} is 
less than or equal to the multiplicity of $V$ in the composition
series of $\WhL$.
This can be shown by induction on the multiplicity.  
As sated in the proof of the previous corollary, 
this multiplicity is equal to $|W|$. 

In order to show the opposite estimate, we recall a result of 
Goodman-Wallach in \cite{GW}. 
Let $\eta'$  be a non-degenerate unitary character of $N$. 
Suppose 
$\sigma' \in \widehat{M}$ and $\nu \in C^{+}$ is generic. 
They obtained 
the ``Harish-Chandra expansion'' of the Jacquet integral (\cite{Jac})
\begin{align*}
& 
J_{\sigma', \nu} 
: 
\Ind_{MA\overline{N}}^{G}
(\sigma' \otimes e^{-\nu-\rho}
\otimes 1_{\overline{N}})_{K} 
\rightarrow 
\stWh{\eta'}{-\nu}{}, 
& 
& 
J_{\sigma', \nu}(f)(g) 
= 
\int_{N} f(g n) \eta'(n)\, dn.  
\end{align*}
Theorem~7.10 of \cite{GW} implies that 
the ``leading term'' of $J_{\sigma', \nu}(f)$ for 
$f \in 
\Ind_{MA\overline{N}}^{G}
(\sigma' \otimes e^{-\nu-\rho}
\otimes 1_{\overline{N}})_{K}$ 
is given by 
\begin{equation}\label{equation:leading term of Jacquet integral} 
A_{w_{0}}(\sigma', \nu)(f)
\in 
\Ind_{MAN}^{G}
(\sigma' \otimes e^{-\nu+\rho} \otimes 1_{N})_{K}, 
\end{equation}
where 
\[
A_{w_{0}}(\sigma', \nu) : 
\Ind_{MA\overline{N}}^{G}
(\sigma' \otimes e^{-\nu-\rho}
\otimes 1_{\overline{N}})_{K}
\rightarrow 
\Ind_{MAN}^{G}
(\sigma' \otimes e^{-\nu+\rho}
\otimes 1_{N})_{K}
\]
is the usual integral intertwining operator between two principal
series, from positive to negative. 

On the other hand, Matumoto proved 
Theorem~\ref{theorem:Matumoto, comp series} 
by using the boundary value maps 
(\cite{Kashiwara-Oshima}, \cite{Oshima ASPM1984}), 
which are (roughly) the maps
of ``taking leading terms'' of Whittaker functions. 
Especially, there is a surjective $\brgK$-homomorphism (see \cite{M0}) 
\begin{align} 
& 
\beta:  
\stWh{\eta'}{-\nu}{\circ} 
= \stWh{\eta'}{-w_{0}\nu}{\circ} 
\twoheadrightarrow 
\Ind_{AN}^{G}
(e^{-\nu+\rho} \otimes 1_{N})_{K}.
\label{equation:boundary value map} 
\end{align}
By composing these maps, we obtain a $\brgK$-homomorphism  
\begin{align}
\beta \circ J_{\sigma', \nu}: 
& 
\Ind_{MA\overline{N}}^{G}
(\sigma' \otimes e^{-\nu-\rho}
\otimes 1_{\overline{N}})_{K} 
\rightarrow 
\Ind_{MAN}^{G}
(\sigma' \otimes e^{-\nu+\rho}
\otimes 1_{N})_{K}, 
\label{eq:boundary circ Jacquet}
\\
& 
\beta\circ J_{\sigma', \nu}(f) 
= 
A_{w_{0}}(\sigma', \nu)(f). 
\notag
\end{align}
By the analyticity of the Jacquet integral and the fact that
\eqref{equation:boundary value map} is well-defined for non-generic
$\nu$, 
this formula is valid for
$\nu \in C^{+}$ which is not necessarily generic. 

Now return to the proof of the lemma. 
We want to show that the dimension of \eqref{equation:Hom WhL to PS}
is at least $|W|$. 
For this, we use Corollary~\ref{corollary:WhL as injective module}(1): 
\[
\stWh{-\eta}{-\Lambda}{\circ}
\simeq 
\bigoplus_{X \in \mathcal{H}_{G}[-\Lambda]_{\mathrm{irr}}, 
\dim(X) = \dim \lie{n}} 
E(X)^{\oplus c_{\dim \lie{n}}(X)}. 
\]
Let $V$ be the Langlands submodule of 
$\Ind_{MAN}^{G}
(\sigma^{c} \otimes e^{-\Lambda+\rho} \otimes 1_{N})_{K}$. 
By Corollary~\ref{corollary:multi of E(Xj) in WhL}, 
it suffices to show that 
\begin{equation}
\label{eq:Hom(E(X),Ind)}
\dim 
\Hom_{\gK}(E(X), 
\Ind_{MAN}^{G}
(\sigma^{c} \otimes e^{-\Lambda+\rho} \otimes 1_{N})_{K}) 
\geq 1 
\end{equation} 
for every irreducible large module $X \in \mathcal{H}_{G}[-\Lambda]$ 
which contains $V$ as a composition factor. 
Fix such $X$. 
By Theorem~\ref{theorem:irred sub of WhLmg}(1), there exists a
non-degenerate character $\eta'$ of $N$ such that $E(X)$ is a direct
summand of $\stWh{\eta'}{-\Lambda}{}$. 
Consider the map \eqref{eq:boundary circ Jacquet} for 
$\sigma'=\sigma^{c}$ and $\nu = \Lambda$. 
Then by the Langlands classification, 
the map \eqref{eq:boundary circ Jacquet} is non-zero on the
Langlands quotient module $V$ of 
$\Ind_{MA\overline{N}}^{G}
(\sigma^{c} 
\otimes e^{-\Lambda-\rho} 
\otimes 1_{\overline{N}})_{K}$, 
so $\beta$ is non-zero on $E(X)$ 
and \eqref{eq:Hom(E(X),Ind)} holds. 
This completes the proof. 
\end{proof}

For the proof of Lemma~\ref{lemma:translation of duWhL from wall} below, 
we need the following estimate. 
\begin{lemma}\label{lemma:sphePS to Pi0}
Suppose $G$ is a split group. 
Denote by $1_{M}$ the trivial representation of $M$. 
If $\Lambda \in \overline{C^{+}}$, then 
\[
\dim 
\Hom_{\gK}(\Ind_{MAN}^{G}(1_{M} \otimes e^{\Lambda+\rho} \otimes 1_{N})_{K}, 
\duWhL) 
\leq 
|W|. 
\]
\end{lemma}
\begin{proof}
By Theorem~5.2.1 of \cite{Kos} and the discussion after it, 
if $\Lambda \in \overline{C^{+}}$, 
then $\Ind_{MAN}^{G}(1_{M} \otimes e^{\Lambda+\rho} \otimes 1_{N})_{K}$ 
is a cyclic $\Ug$-module generated by a $K$-fixed vector 
$f_{0} \in \Ind_{MAN}^{G}(1_{M} \otimes e^{\Lambda+\rho} \otimes 1_{N})_{K}$. 
Strictly speaking, in the paper \cite{Kos}, 
$G$ is assumed to be a real semisimple linear Lie group, 
but since the center of $G$ is contained in $MA$ and $MA$ meets every
component of $G$, 
this theorem is valid for our situation. 
It follows that 
\begin{align*}
\Hom_{\gK}&
(\Ind_{MAN}^{G}(1_{M} \otimes e^{\Lambda+\rho} \otimes 1_{N})_{K}, 
\duWhL) 
\\
& \subset 
\Hom_{K}(\C f_{0}, \res_{\gK}^{\kK} \duWhL). 
\end{align*}
 
On the other hand, by Proposition~\ref{proposition:Osborne?}, 
\begin{align*}
\res_{\gK}^{\kK} 
\duWhL 
&\simeq 
C(K) \otimes_{\Uk} \Ug \otimes_{\Zg \otimes \Un} \C_{\Lambda,\eta} 
\\
&\simeq 
C(K) \otimes_{\Uk} \Uk \otimes E \otimes \Zg \otimes \Un 
\otimes_{\Zg \otimes \Un} \C_{\Lambda,\eta} 
\\
&\simeq 
C(K) \otimes E \otimes \C_{\Lambda,\eta} 
\end{align*}
as $K$-modules. 
Therefore, 
\begin{align*}
\Hom_{\gK}&(\Ind_{MAN}^{G}(1_{M} \otimes e^{\Lambda+\rho} \otimes 1_{N})_{K}, 
\duWhL) 
\\
&\subset 
\Hom_{K}(\C f_{0}, C(K) \otimes E \otimes \C_{\Lambda,\eta})
\\
& \simeq 
\Hom_{\C}(\C, E),
\end{align*}
since the multiplicity of the trivial representation of $K$ in $C(K)$
is one. 
The lemma is the dimension inequality of this inclusion. 
\end{proof}

Let us consider the translation 
$\psi_{\Lambda}^{\Lambda'}$ of $\duWhL$ from singular to non-singular. 

Let $F$ be the finite dimensional irreducible $G$-module 
of which $\Lambda' - \Lambda$ is an extremal weight. 
In the terminology of \cite{Kos}, 
$\indWhL \otimes F$ is $\eta$-finite (Theorem~4.6 there), 
and so is the $\Ug$-submodule $\psi_{\Lambda}^{\Lambda'}(\indWhL)$ of
it. 
Therefore, Theorem~4.4 of \cite{Kos} implies that the space of
Whittaker vectors 
\[
V^{(0)} := \Wh_{\eta}(\psi_{\Lambda}^{\Lambda'}(\indWhL))
\]
is a finite dimensional $\Zg \otimes \Un$-module 
and, as a $\Ug$-module, 
$\psi_{\Lambda}^{\Lambda'}(\indWhL)$ is
generated by $V^{(0)}$. 
This module is non-zero by Theorem~4.6 of \cite{Kos}. 
Note that $V^{(0)}$ is the space of Whittaker vectors 
in $\indWhL \otimes F$ on which $\Zg$ acts by the {\it generalized} 
infinitesimal character $\Lambda'$. 

\begin{lemma}\label{lemma:translation of duWhL from wall}
Let $G$ be a split group.  
Suppose $\Lambda \in \overline{C^{+}}$ and $\Lambda' \in C^{+}$ 
satisfy $\Lambda' - \Lambda \in \mathcal{L}$.  
Then 
\[
S := 
\{u \in \psi_{\Lambda}^{\Lambda'}(\duWhL) 
\mid 
z u = \chi_{\Lambda'}(z) u \ (\forall z \in \Zg)\}
\]
is isomorphic to 
$\duWh{\Lambda'}{\eta}$. 
\end{lemma}
\begin{proof}
By \eqref{equation:psi and Pi}, the definition of Bernstein functor
$\Pi_{\lie{g},\{e\}}^{\gK}$ and Proposition~\ref{proposition:Osborne?}, 
\begin{equation}\label{equation:pf of main lemma-1}
\begin{split}
\psi_{\Lambda}^{\Lambda'}(\duWhL) 
&\simeq 
\Pi_{\lie{g},\{e\}}^{\gK}(\psi_{\Lambda}^{\Lambda'}(\indWhL)) 
\simeq 
\Pi_{\lie{g},\{e\}}^{\gK}(\Ug\, V^{(0)}) 
\\
&\simeq 
R(\gK) \otimes_{\Zg \otimes \Un} V^{(0)} 
\simeq 
C(K) \otimes \Ug \otimes_{\Zg \otimes \Un} V^{(0)}
\\
& \simeq 
C(K) \otimes \Uk \otimes E \otimes \Zg \otimes \Un 
\otimes_{\Zg \otimes \Un} V^{(0)}
\\
&\simeq 
C(K) \otimes E \otimes V^{(0)}. 
\end{split}
\end{equation}
Define a non-zero space 
\[
V^{(1)} 
:=
\{v \in V^{(0)} \mid 
(z-\chi_{\Lambda'}(z)) v = 0 \ \forall z \in \Zg\}. 
\]
Let $d$ and $d'$ be the dimensions of $V^{(0)}$ and $V^{(1)}$,
respectively. 
Choose a basis $v_{1}, v_{2}, \dots, v_{d}$ of $V^{(0)}$ such that 
$v_{1}, v_{2}, \dots, v_{d'}$ is a basis of $V^{(1)}$. 

Assume that $u \in S$. 
Take a basis $\{p_{j}\}_{j}$ of $C(K) \otimes E$.  
Then by \eqref{equation:pf of main lemma-1}, 
there exist constants
$c_{i,j}$ ($i=1,2, \dots, d$, $j=1,2,\dots$) such that 
\begin{equation}\label{equation:pf of main lemma-2} 
u = \sum_{i=1}^{d} \sum_{j} c_{i,j} p_{j} v_{i}.
\end{equation}
Since $\Zg$ is central in $R(\gK)$, 
\begin{align*}
0 
&= (z-\chi_{\Lambda'}(z)) u 
= 
\sum_{i=1}^{d} \sum_{j\geq 1} c_{i,j} p_{j}(z-\chi_{\Lambda'}(z))v_{i}
\\
&=
\sum_{j\geq 1} p_{j} 
\left((z-\chi_{\Lambda'}(z)) \sum_{i=d'+1}^{d} c_{i,j} v_{i}\right). 
\end{align*}
Note that $v_{i}$ ($1 \leq i \leq d'$) are annihilated by 
$z-\chi_{\Lambda'}(z)$ since they are elements of $V^{(1)}$. 
By the linear independence of $\{p_{j}\}_{j}$, 
\[
(z-\chi_{\Lambda'}(z)) \sum_{i=d'+1}^{d} c_{i,j} v_{i}
= 0 
\qquad 
\mbox{for any } j \mbox{ and for any } z\in \Zg,  
\]
which implies that the sum in the left hand side is contained
$V^{(1)}$. 
By the choice of the basis $\{v_{i}\}$, we have 
\[\sum_{i=d'+1}^{d} c_{i,j} v_{i}
= 0 
\qquad 
(\mbox{for any } j),  
\]
and finally by the linear independence of $\{v_{i}\}$, 
\[
c_{i,j} = 0 \qquad \mbox{ for any $j$ and any $i=d'+1,\dots,d$.}
\]
Therefore, 
$u$ is contained in $(C(K) \otimes E) V^{(1)}$ 
so $S \simeq (C(K) \otimes E) V^{(1)}$,  
and we have 
\begin{align*}
S & \simeq  
\Pi_{\lie{g},\{e\}}^{\gK}(\Ug \otimes_{\Zg \otimes \Un} V^{(1)}) 
\\
& = 
\Pi_{\lie{g},\{e\}}^{\gK} 
\left(
\sum_{i=1}^{d'} 
\Ug \otimes_{\Zg \otimes \Un} \C v_{i} 
\right) 
\\
& \simeq 
\Pi_{\lie{g},\{e\}}^{\gK} 
\left(\bigoplus_{i=1}^{d'} \indWh{\Lambda'}{\eta}\right) 
\simeq 
\Pi_{\lie{g},\{e\}}^{\gK}(\indWh{\Lambda'}{\eta})^{\oplus d'} 
\end{align*}
since every $v_{i} \in V^{(1)}$ ($i=1,2,\dots,d'$) generates
$Y_{\Lambda',\eta}$ as a $\Ug$-module by Theorem~4.4 of \cite{Kos}. 
We have shown that 
$S 
\simeq 
\Pi_{\lie{g},\{e\}}^{\gK}(\indWh{\Lambda'}{\eta})^{\oplus d'}$. 
If we show $d'=1$, then the proof is done. 

Let $X_{\Lambda'}$ be a $\brgK$-module admitting the infinitesimal
character $\Lambda'$. 
Since the image of $X_{\Lambda'}$ by a $\brgK$-homomorphism 
in 
$\Hom_{\gK}(X_{\Lambda'}, 
\psi_{\Lambda}^{\Lambda'}(\duWhL))$ 
is
contained in the space $S$, we have 
\begin{align*}
\Hom_{\gK}&(\psi_{\Lambda'}^{\Lambda}(X_{\Lambda'}), \duWhL)
\\
&\simeq 
\Hom_{\gK}(X_{\Lambda'}, 
\psi_{\Lambda}^{\Lambda'}(\duWhL))
\\
&\simeq 
\Hom_{\gK}(X_{\Lambda'}, S)
\simeq 
\Hom_{\gK}(X_{\Lambda'}, 
\Pi_{\lie{g}, \{e\}}^{\gK}(\indWh{\Lambda'}{\eta})^{\oplus d'}). 
\end{align*}
By this, 
\begin{equation}
\label{eq:proof of Lem6.6-1}
\dim 
\Hom_{\gK}(\psi_{\Lambda'}^{\Lambda}(X_{\Lambda'}), \duWhL)
= 
d' \times 
\dim 
\Hom_{\gK}(X_{\Lambda'}, 
\Pi_{\lie{g}, \{ e\}}^{\gK}(\indWh{\Lambda'}{\eta})). 
\end{equation}
By the assumption $\Lambda \in \overline{C^{+}}$ 
and $\Lambda' \in C^{+}$ of this lemma, $\psi_{\Lambda'}^{\Lambda}$ is
a translation functor from non-singular to singular. 
We can choose $\sigma \in \widehat{M}$ so that 
\[
\psi_{\Lambda'}^{\Lambda} 
(\Ind_{MAN}^{G}(\sigma \otimes e^{\Lambda'+\rho} \otimes 1_{N})_{K}) 
\simeq 
\Ind_{MAN}^{G}(1_{M} \otimes e^{\Lambda+\rho} \otimes 1_{N})_{K}. 
\]
Put 
$X_{\Lambda'} 
:= \Ind_{MAN}^{G}(\sigma \otimes e^{\Lambda'+\rho} \otimes
1_{N})_{K}$. 
Then 
\begin{align}
\dim 
\Hom_{\gK}&(\psi_{\Lambda'}^{\Lambda}(X_{\Lambda'}), \duWhL)
\notag\\
&= 
\dim
\Hom_{\gK}
(\Ind_{MAN}^{G}(1_{M} \otimes e^{\Lambda+\rho} \otimes 1_{N})_{K}, 
\duWhL)
\notag\\
&\leq |W|
\label{eq:proof of Lem6.6-2}
\end{align}
by Lemma~\ref{lemma:sphePS to Pi0}, 
and 
\begin{align}
\dim 
\Hom_{\gK}&(X_{\Lambda'}, 
\Pi_{\lie{g}, \{e\}}^{\gK}(\indWh{\Lambda'}{\eta}))
\notag\\
&= 
\Hom_{\gK}
(\Ind_{MAN}^{G}(\sigma \otimes e^{\Lambda'+\rho} \otimes 1_{N})_{K}, 
\Pi_{\lie{g}, \{e\}}^{\gK}(\indWh{\Lambda'}{\eta}))
\notag\\
&= |W|
\label{eq:proof of Lem6.6-3}
\end{align}
by Lemma~\ref{lemma:genPS to Pi0-2}. 
It follows that 
\begin{align*}
|W| &\underset{\eqref{eq:proof of Lem6.6-2}}{\geq}
\dim 
\Hom_{\gK}(\psi_{\Lambda'}^{\Lambda}(X_{\Lambda'}), \duWhL)
\\
&\underset{\eqref{eq:proof of Lem6.6-1}}{=} 
d' \times 
\dim 
\Hom_{\gK}(X_{\Lambda'}, 
\Pi_{\lie{g}, \{e\}}^{\gK}(\indWh{\Lambda'}{\eta}))
\underset{\eqref{eq:proof of Lem6.6-3}}{=}
d' |W|, 
\end{align*}
and $d'$ must be one. This completes the proof. 
\end{proof}
\begin{corollary}\label{corollary:hom, translation from wall}
Suppose $G$ is a split group 
and $\Lambda \in \overline{C^{+}}$, $\Lambda' \in C^{+}$ 
satisfy $\Lambda' - \Lambda \in \mathcal{L}$. 
For $X \in \mathcal{H}_{G}[\Lambda']^{(1)}$, 
\[
\Hom_{\gK}(\psi_{\Lambda'}^{\Lambda}(X), 
\duWhL) 
\simeq 
\Hom_{\gK}(X, \duWh{\Lambda'}{\eta}). 
\]
\end{corollary}
\begin{proof}
Since $X$ admits the infinitesimal character $\Lambda'$, 
the image of an element of 
$\Hom_{\gK}(X, 
\psi_{\Lambda}^{\Lambda'}(\duWhL)$ is contained in $S$ 
defined in Lemma~\ref{lemma:translation of duWhL from wall}. 
Therefore, 
\begin{align*}
\Hom_{\gK}(\psi_{\Lambda'}^{\Lambda}(X), 
\duWhL) 
&\simeq 
\Hom_{\gK}(X, 
\psi_{\Lambda}^{\Lambda'}(\duWhL)) 
\\
&\simeq 
\Hom_{\gK}(X, S) 
\\
&\simeq 
\Hom_{\gK}(X, \duWh{\Lambda'}{\eta}) 
\quad \mbox{by Lemma~\ref{lemma:translation of duWhL from wall}}. 
\end{align*} 
\end{proof}
\begin{proposition}\label{proposition:sub of duWhL} 
Suppose $G$ is a split group. 
\begin{enumerate}
\item
If an irreducible module $X \in \mathcal{H}_{G}[\Lambda]^{(1)}$ 
is a submodule of $\duWhL$, then $X$ is large. 
\item
Suppose $\Lambda' \in C^{+}$. 
Let $X' \in \mathcal{H}_{G}[\Lambda']^{(1)}$ be 
an irreducible large module 
and let $\{X'(\Lambda'+\mu)\}_{\mu}$ be the coherent family 
based on $X'$. 
Then as far as $\Lambda'+\mu$ is contained 
in $\overline{C^{+}}$, 
the multiplicity of the module 
$X'(\Lambda'+\mu)$ in the socle of $\duWh{\Lambda'+\mu}{\eta}$ 
does not depend on $\mu$. 
\end{enumerate}
\end{proposition}
\begin{proof}
(1) We may assume $\Lambda \in \overline{C^{+}}$ 
(cf. Remark~\ref{remark:may assume dominant}). 
First, we show (1) when $\Lambda \in C^{+}$. 
Since $G$ is a split group, 
$X$ is large if and only if the $\tau$-invariant of it 
is empty. 
If $X$ is not large, 
then there exists $\mu \in \mathcal{L}$ such that 
$\Lambda + \mu \in \overline{C^{+}}$ 
and $\psi_{\Lambda}^{\Lambda+\mu}(X) = 0$. 
Apply Corollary~\ref{corollary:hom, translation from wall} 
for the case when $\Lambda$ is $\Lambda+\mu$ 
and $\Lambda'$ is $\Lambda$. 
Then 
$\Hom_{\gK}(X, \duWhL) 
\simeq \Hom_{\gK}
(\psi_{\Lambda}^{\Lambda+\mu}(X), \duWh{\Lambda+\mu}{\eta})
= 0$. 
This means that an irreducible submodule $\duWhL$ must be large. 

Consider the case when $\Lambda \in \overline{C^{+}} - C^{+}$. 
Choose $\mu \in \mathcal{L}$ so that $\Lambda+\mu \in C^{+}$. 
By Theorem~6.18 of \cite{Speh-Vogan}, 
there exists an irreducible module 
$X' \in \mathcal{H}_{G}[\Lambda+\mu]$ 
such that $\psi_{\Lambda+\mu}^{\Lambda}(X')=X$. 
Note that $X'$ is large if and only if $X$ is large. 
Since 
\begin{align*}
\Hom_{\gK}(X, \duWhL)
&\simeq 
\Hom_{\gK}(\psi_{\Lambda+\mu}^{\Lambda}(X'), 
\duWhL) 
\\
&\simeq 
\Hom_{\gK}(X', \duWh{\Lambda+\mu}{\eta})
\end{align*}
by Corollary~\ref{corollary:hom, translation from wall}, 
the claim (1) for this $\Lambda$ follows from the first part of this
proof.

(2) 
This is clear from 
Corollary~\ref{corollary:hom, translation from wall}. 
\end{proof}


\section{Main theorems}

In this section, we show the main theorems on the self-adjoint
properties of our modules $\WhL$, $\WhLmg$ and $E(X)$ for an
irreducible large module $X$.

For $\Lambda \in \overline{C^{+}}$, 
let $R(\Lambda)$ be the integral root system 
(cf. \cite[Definition~7.2.16]{Vogan Green}). 
Let $R^{+}(\Lambda) = R(\Lambda) \cap \Delta^{+}$ be 
its positive system compatible with $\Delta^{+}$ which defines 
$C^{+}$ (cf. Definition~\ref{definition:Weyl chamber}) 
and let $\Phi(\Lambda)$ be the corresponding set of simple roots. 
For an irreducible module $X$ of 
$\mathcal{H}_{G}[\Lambda]$, 
its $\tau$-invariant $\tau(X)$ is a subset of $\Phi(\Lambda)$ and $X$
is large if and only if $\tau(X) = \emptyset$.

\begin{proposition}\label{proposition:dim of emb to duWhL is |W|}
Let $G$ be a split group and assume $G = G_{\max}$. 
Suppose $\Lambda \in \overline{C^{+}}$. 
If $X \in \mathcal{H}_{G}[\Lambda]$ is an irreducible large
module, then 
\begin{equation}
\dim \Hom_{\gK}(X, \duWhL) = |W|. 
\end{equation}
\end{proposition}
\begin{proof}
Because of Proposition~\ref{proposition:sub of duWhL} (2), 
it suffices to show this proposition when $\Lambda$ is contained in 
$C^{+}$, so we assume it. 

If the integral root system $R(\Lambda)$ is empty, 
then any principal series modules are irreducible 
(\cite[Theorem~1.1]{Speh-Vogan})
and $X$ is a principal series module. 
In this case, this proposition is nothing but 
Lemma~\ref{lemma:genPS to Pi0-2}. 

Consider the case when $R(\Lambda)$ is not empty. 
By the Harish-Chandra's subquotient theorem, 
there exists $\sigma \in \widehat{M}$ such that 
$X$ is a composition factor of 
$\Ind_{MAN}^{G}(\sigma \otimes e^{\Lambda+\rho} \otimes 1_{N})_{K}$. 
Choose an infinitesimal character 
$\Lambda+\mu$ from $(\Lambda+\mathcal{L}) \cap \overline{C^{+}}$ 
so that $\Lambda+\mu$ is
most singular in it, that is, 
$\langle \alpha, \Lambda+\mu \rangle = 0$ 
for any $\alpha \in \Phi(\Lambda)$. 
We know that 
\[U:=
\psi_{\Lambda}^{\Lambda+\mu}
(\Ind_{MAN}^{G}(\sigma \otimes e^{\Lambda+\rho} \otimes 1_{N})_{K})
\]
is again a principal series since $\psi_{\Lambda}^{\Lambda+\mu}$ is a
translation functor from non-singular to singular. 

By Theorem~7.229 of \cite{KV}, $\psi_{\Lambda}^{\Lambda+\mu}(X)$ is
irreducible, and it is large. 
But by the choice of $\Lambda+\mu$, 
the composition factors of the principal series module 
$U$ are all large. 
On the other had, since $G = G_{\max}$, 
every principal series module 
has a unique large composition factor and the multiplicity 
of it in this principal series is one 
(\cite[Corollary~6.7]{Vogan GK-dim}). 
Therefore, 
$U$ consists of just one irreducible large module 
$\psi_{\Lambda}^{\Lambda+\mu}(X)$. 
Therefore, 
\begin{align*}
\Hom_{\gK}&(X, \duWhL) 
\\
&\simeq 
\Hom_{\gK}
(X, \psi_{\Lambda+\mu}^{\Lambda}(\duWh{\Lambda+\mu}{\eta}) 
\quad 
\mbox{by Lemma~\ref{lemma:translation of duWhL from wall}}
\\
& \simeq 
\Hom_{\gK}
(\psi_{\Lambda}^{\Lambda+\mu}(X), \duWh{\Lambda+\mu}{\eta}) 
\\
& =
\Hom_{\gK}
(U, 
\duWh{\Lambda+\mu}{\eta})
\\
& =
\Hom_{\gK}
(\psi_{\Lambda}^{\Lambda+\mu}
(\Ind_{MAN}^{G}(\sigma \otimes e^{\Lambda+\rho} \otimes 1_{N})_{K}), 
\duWh{\Lambda+\mu}{\eta})
\\
& \simeq 
\Hom_{\gK}(
\Ind_{MAN}^{G}(\sigma \otimes e^{\Lambda+\rho} \otimes 1_{N})_{K}, 
\duWhL) 
\end{align*}
by Corollary~\ref{corollary:hom, translation from wall}. 
The dimension of the last space is $|W|$ 
by Lemma~\ref{lemma:genPS to Pi0-2}. 
\end{proof}

Assume $G = G_{\max}$. 
Suppose  $X \in \mathcal{H}_{G}[\Lambda]$ is irreducible large. 
By the subquotient theorem, $X$ is a composition factor of 
a principal series module, say $U'$. 
Since $G = G_{\max}$, $X$ is the unique large composition factor of
$U'$ by \cite[Corollary~6.7]{Vogan GK-dim}. 
On the other hand, the dimension of the space of
intertwining operators from $U'$ to $\WhL$ is
$|W|$ by Theorem~I of \cite{Kos}. 
Therefore, the constant $c_{\dim N}(X)$ 
in Corollary~\ref{corollary:WhL as injective module} (1) is $|W|$, 
that is 
\begin{equation}
\label{eq:Cor2.8(1) again}
\WhL 
\simeq 
\bigoplus_{X \in \mathcal{H}_{G}[\Lambda]_{\mathrm{irr}}, 
\Dim(X) = \dim \lie{n}} 
E(X)^{\oplus |W|}, 
\end{equation}
so we have 
\begin{equation}\label{equation:decomposition of the dual-1}
\duWh{-\Lambda}{-\eta} 
\simeq 
(\WhL)^{c} 
\simeq 
\bigoplus_{X \in \mathcal{H}_{G}[\Lambda]_{\mathrm{irr}}, 
\Dim(X) = \dim \lie{n}} 
(E(X)^{c})^{\oplus |W|}
\end{equation}
by 
Proposition
~\ref{proposition:realization of dual Whittaker gK module}(1). 
On the other hand, 
\begin{equation}\label{equation:decomposition of the dual-2}
\soc \duWh{-\Lambda}{-\eta} 
\simeq 
\bigoplus_{X' \in \mathcal{H}_{G}[-\Lambda]_{\mathrm{irr}}, 
\Dim(X') = \dim \lie{n}} 
(X')^{\oplus |W|}
\end{equation}
by Propositions~\ref{proposition:sub of duWhL}(1) and 
\ref{proposition:dim of emb to duWhL is |W|}. 
By the map 
$\mathcal{H}_{G}[\Lambda] 
\ni 
V \mapsto V^{c} 
\in 
\mathcal{H}_{G}[-\Lambda]$ 
of taking contragredient module, 
the number of the irreducible large modules in 
$\mathcal{H}_{G}[\Lambda]$ and that in 
$\mathcal{H}_{G}[-\Lambda]$ are the same. 
By counting the number of irreducible factors in 
the socle of \eqref{equation:decomposition of the dual-1} and 
\eqref{equation:decomposition of the dual-2}, 
we know that the socle of every $E(X)^{c}$ in the right hand side of 
\eqref{equation:decomposition of the dual-1} consists of 
a unique irreducible large module. 

Let $X''$ be the socle of $E(X)^{c}$. 
By the injectivity of $E(X'')$, there exists a $\brgK$-homomorphism
$\varphi_{E(X)} : E(X)^{c} \rightarrow E(X'')$ such that 
\begin{equation}
\label{eq:injectivity discussion on X''}
\begin{xy}
(0,0)*{0}="A0", 
(15,0)*{X''}="A1",
(30,0)*{E(X)^{c}}="A2", 
(22.5,-12)*{E(X'')}="B1", 
\ar "A0";"A1"
\ar "A1";"A2"
\ar "A1";"B1"
\ar "A2";"B1"^{\ \varphi_{E(X)}}
\end{xy}
\end{equation}
is commutative. 
This homomorphism is injective since $X''$ is the unique irreducible
submodule of both $E(X)^{c}$ and $E(X'')$. 
We get an injective $\brgK$-homomorphism 
\[
\bigoplus_{X \in \mathcal{H}_{G}[\Lambda]_{\mathrm{irr}}, 
\Dim(X) = \dim \lie{n}} 
(E(X)^{c})^{\oplus |W|}
\rightarrow 
\bigoplus_{X'' \in \mathcal{H}_{G}[-\Lambda]_{\mathrm{irr}}, 
\Dim(X'') = \dim \lie{n}} 
E(X'')^{\oplus |W|}.
\]
Both hand sides of this are equivalent as $K$-representations 
by Corollary~\ref{corollary:character of WhL}, 
so this map is an isomorphism. 
Thus the map $\varphi_{E(X)}$ 
in \eqref{eq:injectivity discussion on X''} is isomorphism: 
\begin{equation}\label{equation:duality for G=Gmax}
\varphi_{E(X)} : E(X)^{c} \simeq E(X'') 
\end{equation}
when $G = G_{\max}$. 

Finally, consider the case when $G$ does not satisfy $G=G_{\max}$. 
We use the group $G^{+}$ defined 
in \S~\ref{section:global characters}. 
Let $X$ be an irreducible $(\lie{g}, K^{+})$-module and 
\begin{equation}
\label{eq:res of X to gK}
\res_{\lie{g}, K^{+}}^{\gK}(X) 
\simeq 
\bigoplus_{i=1}^{l'} X_{i} 
\end{equation}
be its irreducible decomposition as an $\brgK$-module. 
For $G^{\ast} = G$ or $G^{+}$, 
we denote the injective envelope of 
$Y \in \mathcal{H}_{G^{\ast}}[\Lambda]^{(1)}$ by 
$E_{G^{\ast}}(Y)$. 
By the injectivity, 
$\Hom_{\lie{g}, K^{+}}(\, \ast\, , E_{G^{+}}(X))$ is an exact
functor on $\mathcal{H}_{G^{+}}[\Lambda]^{(1)}$. 
Since 
\[
\Hom_{\gK}(\, \ast\, , \res_{\lie{g}, K^{+}}^{\gK}(E_{G^{+}}(X)))
\simeq 
\Hom_{\lie{g}, K^{+}}
(\mathrm{induced}_{\gK}^{\lie{g}, K^{+}}(\, \ast\, ), 
E_{G^{+}}(X))
\]
and 
$\mathrm{induced}_{\gK}^{\lie{g}, K^{+}}$ is an exact functor from 
$\mathcal{H}_{G}[\Lambda]$ to $\mathcal{H}_{G^{+}}[\Lambda]$ which 
preserves the infinitesimal character 
(see Propostion~2.77 and (2.74c) of \cite{KV}), 
$\res_{\lie{g}, K^{+}}^{\gK}(E(X))$ is an injective module 
in the category $\mathcal{H}_{G}[\Lambda]^{(1)}$. 
Therefore by Proposition~\ref{proposition:decomp of injective}(2),
\begin{equation}
\label{eq:res of E(X) to gK} 
\res_{\lie{g}, K^{+}}^{\gK}(E_{G^{+}}(X)) 
\simeq 
\bigoplus_{i=1}^{l'} E_{G}(X_{i}).  
\end{equation}
Since $E_{G}(X_{i})$ is indecomposable, 
\eqref{equation:duality for G=Gmax} and 
\eqref{eq:res of E(X) to gK} 
imply that, even if $G$ does not satisfy $G=G_{\max}$, 
for every $X \in \mathcal{H}_{G}[\Lambda]$, 
there exists a unique irreducible large module 
$X'' \in \mathcal{H}_{G}[-\Lambda]$ 
such that 
$E_{G}(X)^{c}$ is isomorphic to $E_{G}(X'')$. 

Note that the correspondnce $X \mapsto X''$ is injective. 
In fact, if $E(X'') \simeq E(Y_{i})^{c}$ ($i=1,2$) for 
$Y_{1}, Y_{2} \in \mathcal{H}_{G}[\Lambda]$, 
then 
$E(Y_{1})^{c} \simeq E(Y_{2})^{c}$ so 
$E(Y_{1}) \simeq E(Y_{2})$. 
The socles of both sides are $Y_{1}$ and $Y_{2}$. 

Moreover, 
as is stated below \eqref{equation:decomposition of the dual-2}, 
the map $\mathcal{H}_{G}[-\Lambda] 
\ni X'' \mapsto (X'')^{c} 
\in \mathcal{H}_{G}[\Lambda]$ of taking contragredient is a one to
one correspondence between the set of large irreducible modules in 
$\mathcal{H}_{G}[-\Lambda]$ and that in $\mathcal{H}_{G}[\Lambda]$. 
Put $\widetilde{X} := (X'')^{c}$. 
Then the correspondence 
$\mathcal{H}_{G}[\Lambda] 
\ni X \mapsto \widetilde{X} \in 
\mathcal{H}_{G}[\Lambda]$ given by 
$\varphi_{E(X)} : 
E(X)^{c} \simeq E(\widetilde{X}^{c})$ is a permutation of
the set of large irreducible modules in $\mathcal{H}_{G}[\Lambda]$.

\begin{theorem}\label{theorem:1st duality theorem}
Suppose $G$ is a split group. 
If $X \in \mathcal{H}_{G}[\Lambda]^{(1)}$ is an irreducible large module, 
then its injective envelope $E(X)$ in $\mathcal{H}_{G}[\Lambda]^{(1)}$ 
has a unique irreducible quotient module $\widetilde{X}$, which is large. 
It satisfies 
\[
E(X)^{c} \simeq E(\widetilde{X}^{c}) 
\]
and the correspondence $X \mapsto \widetilde{X}$ is a permutation of
the set of large irreducible modules in
$\mathcal{H}_{G}[\Lambda]^{(1)}$. 
Moreover, if $G = G_{\max}$, then 
$\widetilde{X} \simeq X$. 
\end{theorem}
\begin{proof}
We have already shown this theorem except for the last statement. 
To prove the last statement, we recall the notion of cones of
Harish-Chandra modules. 

On the complexified Grothendieck group 
$K(\mathcal{H}_{G}[\Lambda]) \otimes_{\Z} \C$ of 
$\mathcal{H}_{G}[\Lambda]$, 
we can define the coherent continuation representation 
of the integral Weyl group $W(\Lambda)$ (\cite{Vogan Green}). 
Let $V$ be an irreducible module in $\mathcal{H}_{G}[\Lambda]$. 
The $W(\Lambda)$-submodule of 
$K(\mathcal{H}_{G}[\Lambda]) \otimes_{\Z} \C$ 
generated by the irreducible modules appearing in $w \cdot V$ 
($w \in W(\Lambda)$) is called the cone over $V$ 
and we denote it by $\mathcal{C}(V)$. 
If $V$ is large, we call $\mathcal{C}(V)$ a big cone. 
Suppose $V$ is irreducible large. 
Then by Corollary~7.3.19 of \cite{Vogan Green}, 
every irreducible modules appearing in $\mathcal{C}(V)$ other than $V$
has non-empty $\tau$-invariant. 
Since we are assuming $G$ to be split, 
$V$ is the unique irreducible large module appearing in 
$\mathcal{C}(V)$. 
On the other hand, if $G = G_{\max}$, 
it is known that the number of blocks 
(\cite[Definition~9.2.1]{Vogan Green}) is equal to the number of big
cones in $K(\mathcal{H}_{G}[\Lambda]) \otimes_{\Z} \C$ 
(\cite[Proposition~4.4.5]{M2}). 
It follows that there is just one irreducible large $\brgK$-module in
each block. 

Since $E(X)$ is indecomposable and by the definition of blocks, 
every irreducible composition factor of $E(X)$ 
is contained in the same block as $X$. 
We know that the unique irreducible quotient module $\widetilde{X}$ 
of $E(X)$ is large. 
This must be isomorphic to $X$. 
\end{proof}

For a module $Y \in \mathcal{H}_{G}[\Lambda]^{(1)}$, 
we denote its projective cover in $\mathcal{H}_{G}[\Lambda]^{(1)}$ by $P(Y)$. 
\begin{corollary}\label{corollary:proj hull is inj env}
In the setting of Theorem~\ref{theorem:1st duality theorem}, 
\[
E(X) = P(\widetilde{X}), 
\]
and 
$\WhL$, $\WhLmg$ are projective modules in the category 
$\mathcal{H}_{G}[\Lambda]^{(1)}$. 
\end{corollary}
\begin{theorem}\label{theorem:2nd duality theorem}
Suppose $G$ is a split group. 
Then 
\begin{equation}
\label{eq:2nd duality theorem-1}
(\WhL)^{c} 
\simeq 
\Pi_{\lie{g}, \{e\}}^{\gK}(Y_{-\Lambda, -\eta}) 
\simeq 
\stWh{\eta}{-\Lambda}{\circ}. 
\end{equation}
If moreover $G=G_{\max}$, then 
\[
(\WhLmg)^{c} 
\simeq 
\stWh{\eta}{-\Lambda}{}. 
\]
\end{theorem}
\begin{proof}
First, suppose $G=G_{\max}$. 
We use the notation used 
in Theorem~\ref{theorem:1st duality theorem} and the discussion above
it. 
By \eqref{eq:Cor2.8(1) again} 
and \eqref{equation:decomposition of the dual-1}, 
\begin{align*}
\duWh{-\eta}{-\Lambda}
\simeq 
(\WhL)^{c} 
&
\underset{\eqref{equation:decomposition of the dual-1}}{\simeq} 
\bigoplus_{X \in \mathcal{H}_{G}[\Lambda]_{\mathrm{irr}}, 
\Dim(X) = \dim \lie{n}} 
(E(X)^{c})^{\oplus |W|}
\\
&\simeq 
\bigoplus_{X'' \in \mathcal{H}_{G}[-\Lambda]_{\mathrm{irr}}, 
\Dim(X'') = \dim \lie{n}} 
E(X'')^{\oplus |W|}
\\
&
\underset{\eqref{eq:Cor2.8(1) again}}{\simeq} 
\stWh{\eta}{-\Lambda}{\circ}. 
\end{align*}

Next, since $G=G_{\max}$, 
the number of principal nilpotent $G$-orbits in $\lier{g}$ is one. 
See \cite[p205]{M2}. There, the condition $G=G_{\max}$ is called 
$G$ is of type II. 
It follows that, if $\eta$ is non-degenerate and $X$ is large, 
the condition $\eta \in \WF(X)$ in 
Corollary~\ref{corollary:WhL as injective module}(3) 
is automatically satisfied. 
Therefore, by Corollary~\ref{corollary:WhL as injective module} and 
Theorem~\ref{theorem:1st duality theorem}, 
\begin{align*}
(\WhLmg)^{c} 
& \simeq 
\bigoplus_{X \in \mathcal{H}_{G}[\Lambda]_{\mathrm{irr}}, 
\mathrm{large}}
E(X)^{c}
\simeq 
\bigoplus_{X'' \in \mathcal{H}_{G}[-\Lambda]_{\mathrm{irr}}, 
\mathrm{large}}
E(X'')
\simeq 
\stWh{\eta}{-\Lambda}{}. 
\end{align*}

Finally, let us consider the case when $G=G_{\max}$ is not satisfied. 
Recall the proof of Theorem~\ref{theorem:character of WhLmg}. 
Just in the same way 
as \eqref{equation:decomposition of WhL for Gmax to G}, 
we have 
\[
\res_{\lie{g}, K^{+}}^{\gK} 
\stWh{\eta}{\Lambda}{G^{+}, \circ}
\simeq 
\bigoplus_{i=1}^{l} 
\stWh{y_{i} \cdot \eta}{\Lambda}{G, \circ}, 
\qquad 
G\backslash G^{+} \simeq \{y_{1}, y_{2}, \dots, y_{l}\} 
\subset M^{+}.  
\]
Since Corollary~\ref{corollary:WhL as injective module}(1) 
does not depend on the choice of the non-degenerate unitary chracter
$\eta$, 
$\res_{\lie{g}, K^{+}}^{\gK} 
\stWh{\eta}{\Lambda}{G^{+}, \circ}
\simeq 
(\stWh{\eta}{\Lambda}{G, \circ})^{\oplus l}$. 
Thus the isomorphism \eqref{eq:2nd duality theorem-1} for $G$ 
follows from that for the $G=G_{\max}$ case. 
\end{proof}
\begin{remark}
If $G=G_{\max}$ does not hold, 
$\WhLmg$ {\it does} depend on the choice of $\eta$. 
The problem whether 
$(\WhLmg)^{c}
\simeq 
\stWh{\eta}{-\Lambda}{}$ holds or not is open. 
We will see that this holds for $G=Sp(n,\R)$ 
in Remark~\ref{remark:self-duality in Sp(n,R) case}. 
\end{remark}


\section{Examples}
\label{section:examples}

The composition factors of standard modules are known from 
the Kazhdan-Lusztig-Vogan conjecture. 
By using the results in this paper and some other methods, 
we can determine the socle filtration of the injective envelope $E(X)$
of an irreducible large $\brgK$-module $X$ 
if $G$ is a small split group. 
In this section, we present it 
in the case when $G$ is a real rank two connected split linear group 
and the infinitesimal character $\Lambda$ is non-singular integral.

Before going ahead, we introduce some notation used in this
section. 

We use the Langlands classification to identify irreducible
$\brgK$-modules. 
Let $\gamma$ be a regular character of a Cartan subgroup $H$ of $G$ 
with non-singular infinitesimal character 
(cf. \cite[\S~6.6]{Vogan Green}). 
The corresponding infinitesimal character is denoted by
$\overline{\gamma}$ and 
the (integral) length of $\gamma$ (\cite[Definition~8.1.4]{Vogan Green}) is 
denoted by $l^{I}(\gamma)$. 
We denote the standard $\brgK$-module with parameter $\gamma$ by
$X(\gamma)$, 
and its Langlands subquotient module by $\overline{X}(\gamma)$. 
When we realize the standard module $X(\gamma)$ as a generalized
principal series 
$\Ind_{P'}^{G}(\sigma \otimes e^{\nu+\rho_{P'}} \otimes 1)_{K}$ 
induced from a discrete series $\sigma$ and the character $e^{\nu}$ of
the vector part, 
we set $\nu$ to be positive with respect to the parabolic 
subgroup $P'$. 
So the Langlands subquotient $\overline{X}(\gamma)$ is the unique
irreducible quotient module of $X(\gamma)$. 
We call such $X(\gamma)$ {\it positively induced}.

Let $X$ be a $\brgK$-module. 
A diagram 
\[
X 
\simeq 
\begin{matrix}
V_{1} 
\\
V_{2} \oplus V_{3}
\\
V_{4}
\end{matrix} 
\]
means that the socle of $X$ is $V_{4}$ and 
the socle of $X/V_{4}$ is $V_{2} \oplus V_{3}$ and so on. 
In this section, 
we use such diagrammatic expression of the socle filtration of a
$\brgK$-module. 

Next, we introduce some methods used in this section. 

\subsection{Parity condition.} 
The first one is the parity condition. 
As a result of Kazhdan-Lusztig-Vogan conjecture and 
Theorem~9.5.1 of \cite{Vogan Green}, the following proposition holds: 
\begin{proposition}{\rm (Parity condition \cite[Corollary~4.5]{HTY})}
\label{proposition:parity condition}
Suppose $V$ is a $\brgK$-module of finite length 
which admits a non-singular infinitesimal character. 
If the integral lengths of the irreducible factors in the 
$k$-th floor of the socle filtration of $V$ 
are all even (resp. odd), 
then those of the factors in $(k+1)$-st floor are all odd (resp. even). 
\end{proposition}

\subsection{Multiplicity in the second floor.}
The second one is on the extension group for a discrete series and its
Cayley transform. 

\begin{lemma}
\label{lemma:Ext of DS and its Cayley transf}
Suppose $G$ has a compact Cartan subgroup $H_{c}$. 
Let $\gamma$ be a regular character of $H_{c}$ such that 
$X(\gamma) = \overline{X}(\gamma)$ is a large discrete
series module with non-singular infinitesimal character. 
Choose a positive system $\Delta^{+}$ of the root system 
$\Delta(\lie{g}, \lie{h}_{c})$ so that $\overline{\gamma}$ is
dominant for it. 
For a simple non-compact imaginary root $\alpha \in \Delta^{+}$, 
let $\gamma^{\alpha}$ be the Cayley transform of $\gamma$ by
$\alpha$. 
Let $\gamma'$ be a regular character whose $\tau$-invariant
contains $\alpha$ but $\gamma'$ is not $K$-conjugate to
$\gamma^{\alpha}$. 
Then 
\begin{align*}
& 
\Ext_{\gK}^{1}
(\overline{X}(\gamma^{\alpha}), \overline{X}(\gamma)) 
\simeq \C
& 
& \mbox{and} 
& 
& 
\Ext_{\gK}^{1}
(\overline{X}(\gamma'), \overline{X}(\gamma)) 
= 0.  
\end{align*}
Especially, the multiplicity of $\overline{X}(\gamma^{\alpha})$ 
in the second floor of the socle filtration of 
$E(\overline{X}(\gamma))$ is at most one. 
\end{lemma}
\begin{proof}
Theorem~8.5.18 of \cite{Vogan Green} implies that 
$U_{\alpha}(\overline{X}(\gamma))$ 
(\cite[Definition~7.3.17]{Vogan Green}) contains
$\overline{X}(\gamma^{\alpha})$ with multiplicity one 
and does not contain $\overline{X}(\gamma')$. 
Then by Corollary~9.5.7 of \cite{Vogan Green}, 
\[
\dim \Ext_{\gK}^{1}
(\overline{X}(\gamma^{\alpha}), \overline{X}(\gamma)) 
= 
\dim \Hom_{\gK}
(\overline{X}(\gamma^{\alpha}), U_{\alpha}(\overline{X}(\gamma))) 
\leq 1. 
\]
Moreover, a non-splitting extension of
$\overline{X}(\gamma)$ by $\overline{X}(\gamma^{\alpha})$ 
is realized in $X(\gamma^{\alpha})$. 
Therefore, the equality holds in the above inequality. 
The proof of 
$\Ext_{\gK}^{1}
(\overline{X}(\gamma'), \overline{X}(\gamma)) 
= 0$ 
is the same. 

The final statement is a consequence of this result. 
In fact, let 
\[
0 \to \overline{X}(\gamma) \to E \to \overline{X}(\gamma^{\alpha}) \to 0
\]
be a non-splitting exact sequence. 
It yields a long exact sequence 
\begin{align*}
0 &\to 
\Hom_{\gK}(\overline{X}(\gamma^{\alpha}), \overline{X}(\gamma)) 
\to 
\Hom_{\gK}(\overline{X}(\gamma^{\alpha}), E) 
\to 
\Hom_{\gK}
(\overline{X}(\gamma^{\alpha}), \overline{X}(\gamma^{\alpha})) 
\\
& 
\to 
\Ext_{\gK}^{1}(\overline{X}(\gamma^{\alpha}), \overline{X}(\gamma)) 
\to 
\Ext_{\gK}^{1}(\overline{X}(\gamma^{\alpha}), E) 
\to 
\Ext_{\gK}^{1}
(\overline{X}(\gamma^{\alpha}), \overline{X}(\gamma^{\alpha})). 
\end{align*}
The final term is zero by Theorem~9.5.1 of \cite{Vogan Green}. 
Therefore, from the first part of this lemma, 
we obtain an exact sequence 
\[
0 
\to 
0 
\to 
0 
\to 
\C 
\to 
\C 
\to 
\Ext_{\gK}^{1}(\overline{X}(\gamma^{\alpha}), E) 
\to
0. 
\]
This implies that 
$\Ext_{\gK}^{1}(\overline{X}(\gamma^{\alpha}), E)=0$, 
so there is no non-splitting extension of $E$ 
by $\overline{X}(\gamma^{\alpha})$. 
This means that there is no module which contains a subquotient module 
isomorphic to 
$\begin{matrix}
\overline{X}(\gamma^{\alpha}) \oplus \overline{X}(\gamma^{\alpha}) 
\\
\overline{X}(\gamma) 
\end{matrix}$. 
The final statement of this lemma follows from this. 
\end{proof}

Hereafter, we present examples of the socle filtrations of $E(X)$. 
Among the rank two cases, the case when $G=Sp(2,\R)$ and $X$ is a
discrete series seems to require the most technical ingenuity. 
So we explain this case in detail and for other cases we only present
the final results. 

\


\noindent
\textbf{I. $G=Sp(2,\R)$, block $PSO(3,2)$. }

In the group $Sp(2,\R)$, 
there are four conjugacy classes of Cartan subgroups. 
The fundamental Cartan subgroup $H_{f}$ is a compact Cartan subgroup. 
Between $H_{f}$ and the split Cartan subgroup $H_{s}$, 
there are two Cartan subgroups. 
One is the Cayley transform of the compact Cartan subgroup 
through a long non-compact imaginary root, 
which we denote by $H_{J} = T_{J} A_{J}$. 
The other is the Cayley transform of the compact Cartan subgroup
through a short non-compact imaginary root, 
which we denote by $H_{S} = T_{S} A_{S}$ 
(``J'' stands for Jacobi and ``S'' for Siegel). 

There are four $K$-conjugacy classes of regular characters of $H_{f}$, 
which we call $\gamma_{j}$ ($j=0,1,2,3$). 
These are discrete series modules, and they are contained in the block
$PSO(3,2)$. 
There are eight $K$-conjugacy classes of regular characters 
of $H_{J}$, 
which we denote by $\gamma_{5}$, $\gamma_{6}$, $\gamma_{7}$,
$\gamma_{8}$ 
and $\gamma_{0'}$, $\gamma_{1'}$, $\gamma_{2'}$, $\gamma_{3'}$. 
The regular characters $\gamma_{j}$ ($j=5,6,7,8$) are contained in
the block $PSO(3,2)$ and 
$\gamma_{j'}$ ($j=0,1,2,3$) are contained in the block $PSO(4,1)$. 
There are two $K$-conjugacy classes of regular characters 
of $H_{S}$, 
which we denote by $\gamma_{4}$ and $\gamma_{9}$. 
These are contained in the block $PSO(3,2)$. 
There are four $K$-conjugacy classes of regular characters 
of $H_{s}$, 
which we denote by $\gamma_{10}$, $\gamma_{11}$, $\gamma_{4'}$ 
and $\gamma_{0''}$. 
$\gamma_{10}$ and $\gamma_{11}$ are contained in the block $PSO(3,2)$, 
$\gamma_{4'}$ in the block $PSO(4,1)$ and the block $PSO(5)$ consists
of $\gamma_{0''}$ only. 
The regular characters which correspond to irreducible large 
modules are $\gamma_{0}$, $\gamma_{1}$, $\gamma_{0'}$, $\gamma_{1'}$ 
and $\gamma_{0''}$. 
$\overline{X}(\gamma_{0})$ and $\overline{X}(\gamma_{1})$ 
are the large discrete series modules. 

Since the block $PSO(5)$ consists of $\overline{X}(\gamma_{0''})$, 
its injective envelope is isomorphic to itself; 
$E(\overline{X}(\gamma_{0''})) \simeq \overline{X}(\gamma_{0''})$.

Let $\alpha_{1}$, $\alpha_{2}$ be the long and short simple root,
respectively. 
Table~\ref{table:PSO(3,2)} is the data on the regular characters in
the block $PSO(3,2)$ 
(cf. \cite{Vogan PM40}).  
\begin{table}[h]
\centering
\caption{Data on the regular characters in the block $PSO(3,2)$}
\label{table:PSO(3,2)}
\begin{tblr}{|c|c|c|c|c|c|c|c|c|}
\hline
\SetCell[r=2]{c} 
& \SetCell[r=2]{c} CSG & \SetCell[c=2]{c} simple root 
& & 
\SetCell[c=2]{c} cross action 
& & 
\SetCell[c=2]{c} Cayley transf.  
&
& \SetCell[r=2]{c} length 
\\ \hline
& & $\alpha_{1}$ & $\alpha_{2}$ 
& $s_{1} \times *$ & $s_{2} \times *$ & 
$\alpha_{1}$ & $\alpha_{2}$ & 
\\ \hline 
$\gamma_{0}$ & $H_{f}$ & ncI & ncI & 
$\gamma_{2}$ & $\gamma_{1}$ & $\gamma_{5}$ & $\gamma_{4}$ & $0$
\\ \hline 
$\gamma_{1}$ & $H_{f}$ & ncI & ncI & 
$\gamma_{3}$ & $\gamma_{0}$ & $\gamma_{6}$ & $\gamma_{4}$ & $0$
\\ \hline 
$\gamma_{2}$ & $H_{f}$ & ncI & cpt 
& $\gamma_{0}$ & $\gamma_{2}$ & $\gamma_{5}$ & -- 
& $0$
\\ \hline 
$\gamma_{3}$ & $H_{f}$ & ncI & cpt 
& $\gamma_{1}$ & $\gamma_{3}$ & $\gamma_{6}$ & -- 
& $0$
\\ \hline 
$\gamma_{4}$ & $H_{S}$ & $C^{+}$ & rI 
& $\gamma_{9}$ & $\gamma_{4}$ & -- & $\gamma_{0}$, $\gamma_{1}$ 
& $1$
\\ \hline 
$\gamma_{5}$ & $H_{J}$ & rI & $C^{+}$ 
& $\gamma_{5}$ & $\gamma_{7}$ & $\gamma_{0}$, $\gamma_{2}$ & --
& $1$
\\ \hline 
$\gamma_{6}$ & $H_{J}$ & rI & $C^{+}$ 
& $\gamma_{6}$ & $\gamma_{8}$ & $\gamma_{1}$, $\gamma_{3}$ & --
& $1$
\\ \hline 
$\gamma_{7}$ & $H_{J}$ & ncI & $C^{-}$ 
& $\gamma_{8}$ & $\gamma_{5}$ & $\gamma_{10}$ & --
& $2$
\\ \hline 
$\gamma_{8}$ & $H_{J}$ & ncI & $C^{-}$ 
& $\gamma_{7}$ & $\gamma_{6}$ & $\gamma_{10}$ & --
& $2$
\\ \hline 
$\gamma_{9}$ & $H_{S}$ & $C^{-}$ & ncII 
& $\gamma_{4}$ & $\gamma_{9}$ & -- & $\gamma_{10}$, $\gamma_{11}$ 
& $2$
\\ \hline 
$\gamma_{10}$ & $H_{s}$ & rI & rII 
& $\gamma_{10}$ & $\gamma_{11}$ & $\gamma_{7}$, $\gamma_{8}$ 
& $\gamma_{9}$ 
& $3$
\\ \hline 
$\gamma_{11}$ & $H_{s}$ & rn & rII
& $\gamma_{11}$ & $\gamma_{10}$ & -- & $\gamma_{9}$ 
& $3$
\\ \hline 
\end{tblr}
\end{table}

Here, $C^{+}$ (resp. $C^{-}$) means a complex root $\alpha$ with 
$\theta \alpha$ positive (resp. negative), 
where $\theta$ is the Cartan involution. 
Also, cpt means a compact imaginary root, 
ncI, ncII (resp. rI, rII) mean noncompact imaginary roots 
(resp. real roots) of type I and II, 
and rn means a real root not in the range of Cayley transforms.

The composition factors of $X(\gamma)$ is known by the
Kazhdan-Lusztig-Vogan (KLV) conjecture. 
In the Grothendieck group, 
\begin{align}
&
X(\gamma_{i}) = \overline{X}(\gamma_{i}), 
\quad i=0,1,2,3, 
\notag\\
& 
X(\gamma_{4}) = 
\overline{X}(\gamma_{0}) + \overline{X}(\gamma_{1}) 
+ \overline{X}(\gamma_{4}), 
\notag\\
&
X(\gamma_{5}) = 
\overline{X}(\gamma_{0}) + \overline{X}(\gamma_{2}) 
+ \overline{X}(\gamma_{5}), 
\qquad
X(\gamma_{6}) = 
\overline{X}(\gamma_{1}) + \overline{X}(\gamma_{3}) 
+ \overline{X}(\gamma_{6}), 
\notag\\
&
X(\gamma_{7}) = 
\overline{X}(\gamma_{0}) + \overline{X}(\gamma_{4})
+\overline{X}(\gamma_{5}) + \overline{X}(\gamma_{7}), 
\notag\\
&
X(\gamma_{8}) 
= 
\overline{X}(\gamma_{1})+\overline{X}(\gamma_{4})
+\overline{X}(\gamma_{6})+\overline{X}(\gamma_{8}), 
\notag\\
&
X(\gamma_{9})
= 
\overline{X}(\gamma_{0}) 
+ \overline{X}(\gamma_{1}) + \overline{X}(\gamma_{4})
+ \overline{X}(\gamma_{5}) + \overline{X}(\gamma_{6})
+ \overline{X}(\gamma_{9}), 
\notag\\
&
X(\gamma_{10})
= 
\overline{X}(\gamma_{0}) + \overline{X}(\gamma_{1}) 
+ 2 * \overline{X}(\gamma_{4}) + \overline{X}(\gamma_{5}) 
+ \overline{X}(\gamma_{6}) 
\notag\\
& \hspace{20mm}
+ \overline{X}(\gamma_{7}) 
+ \overline{X}(\gamma_{8}) + \overline{X}(\gamma_{9}) 
+ \overline{X}(\gamma_{10}), 
\notag\\
&
X(\gamma_{11})
= 
\overline{X}(\gamma_{0}) + \overline{X}(\gamma_{1}) 
+ \overline{X}(\gamma_{2}) + \overline{X}(\gamma_{3}) 
\notag\\
& \hspace{20mm}
+ \overline{X}(\gamma_{4}) + \overline{X}(\gamma_{5}) 
+ \overline{X}(\gamma_{6}) + \overline{X}(\gamma_{9}) 
+ \overline{X}(\gamma_{11}). 
\notag
\end{align}

It is known that 
$\overline{X}(\gamma_{0})^{c} 
\simeq 
\overline{X}(\gamma_{1})$, 
$\overline{X}(\gamma_{2})^{c} 
\simeq 
\overline{X}(\gamma_{3})$, 
$\overline{X}(\gamma_{5})^{c} 
\simeq 
\overline{X}(\gamma_{6})$, 
$\overline{X}(\gamma_{7})^{c} 
\simeq 
\overline{X}(\gamma_{8})$, 
and $\overline{X}(\gamma_{j})$ ($j=4,9,10,11$) are self-adjoint. 

Consider the case when $\eta$ is in the wave front set of
$\overline{X}(\gamma_{0})$, 
so this module is the unique irreducible submodule of $\WhLmg$. 
The socle filtrations of positively induced standard modules 
$X(\gamma_{4})$ and $X(\gamma_{5})$ are 
\begin{align*}
& 
X(\gamma_{4}) 
\simeq 
\begin{matrix} 
\overline{X}(\gamma_{4}) 
\\
\overline{X}(\gamma_{0})
\oplus 
\overline{X}(\gamma_{1}) 
\end{matrix}
&
& 
\mbox{and}  
&
& 
X(\gamma_{5}) 
\simeq 
\begin{matrix} 
\overline{X}(\gamma_{5}) 
\\
\overline{X}(\gamma_{0})
\oplus 
\overline{X}(\gamma_{2}) 
\end{matrix}, 
\end{align*}
respectively. 
Since $\eta \in \WF(\overline{X}(\gamma_{0}))$, 
Theorems~\ref{theorem:irred sub of WhLmg}(2) 
and \ref{theorem:exactness of Wh vector}(2) imply that there are
non-zero $\brgK$-homomorphisms from $X(\gamma_{4})$ and
$X(\gamma_{5})$ to $\WhLmg$. 
By the last statement of 
Theorem~\ref{theorem:holomorphy of Jacquet}, 
these homomorphisms are realized by some Jacquet integrals. 
Then by Proposition~\ref{proposition:analy. conti of Jacquet}, there
are $\brgK$-submodules of $\WhLmg$ whose characters are 
$\ch(X(\gamma_{4}))$ and $\ch(X(\gamma_{5}))$. 
Since $\overline{X}(\gamma_{0})$ is the unique irreducible submodule
of $\WhLmg$, 
the parity condition implies that 
the next diagram is contained in the bottom of the socle filtration 
of $E(\overline{X}(\gamma_{0}))$. 
\begin{equation}
\label{eq:Sp2 PSO(3,2)-1}
\begin{matrix}
\overline{X}(\gamma_{1}) 
\oplus 
\overline{X}(\gamma_{2}) 
\\
\overline{X}(\gamma_{4}) 
\oplus 
\overline{X}(\gamma_{5}) 
\\
\overline{X}(\gamma_{0})
\end{matrix}. 
\end{equation}

Consider the standard module $X(\gamma_{10})$ realized by a principal
series positively induced from the character 
$\gamma_{10} 
= (\sigma, \Lambda) \in \widehat{M} \times \lie{a}^{\ast}$ of $H_{s}$. 
Let $P_{J} = M_{J} A_{J} N_{J}$ be the parabolic subgroup of $G$ 
such that the Cartan subgroup $H_{J}=T_{J} A_{J}$ is a Cartan subgroup
of $M_{J} A_{J}$ and it contains the minimal parabolic $P$. 
By induction by stages, 
\begin{equation}
\label{eq:induced expr of X(gamma_10)}
X(\gamma_{10}) 
\simeq 
\Ind_{P_{J}}^{G}
(\Ind_{P \cap M_{J}}^{M_{J}}
(\sigma \otimes e^{\Lambda|_{A \cap M_{J}}+\rho_{1}}
\otimes 1_{N \cap M_{J}})_{M_{J} \cap K} 
\otimes e^{\Lambda|_{A_{J}}+\rho_{1}} 
\otimes 1_{N_{J}})_{K}. 
\end{equation}
Here, $\rho_{1}$ and $\rho_{2}$ are appropriately defined
$\rho$-shift. 
$M_{J}$ is isomorphic to $SL(2,\R) \times \{\pm 1\}$ and then 
$\Ind_{P \cap M_{J}}^{M_{J}}
(\sigma \otimes e^{\Lambda|_{A \cap M_{J}}+\rho_{1}}
\otimes 1_{N \cap M_{J}})_{M_{J} \cap K}$ 
is a reducible principal series of
$SL(2,\R) \times \{\pm 1\}$, 
whose socle consists of two discrete series. 
Since $\gamma_{7}$ is a Cayley transform of $\gamma_{10}$ by the
long simple root, $X(\gamma_{7})$ is the generalized principal series 
induced from one of the two discrete series modules
of $M_{J}$ in the same way as \eqref{eq:induced expr of X(gamma_10)}. 
It follows that $X(\gamma_{7})$ is a submodule of $X(\gamma_{10})$ in
the positively induced picture. 
In the same way, we can check that $X(\gamma_{9})$ is a submodule of
$X(\gamma_{10})$ and $X(\gamma_{11})$. 

By the Langlands classification, $\overline{X}(\gamma_{10})$ is the
unique irreducible quotient of $X(\gamma_{10})$, 
$\overline{X}(\gamma_{11})$ that of $X(\gamma_{11})$, 
$\overline{X}(\gamma_{7})$ that of $X(\gamma_{7})$ and  
$\overline{X}(\gamma_{9})$ that of $X(\gamma_{9})$. 
On the other hand, 
by Theorem~6.2(e) and Corollary~6.7 of \cite{Vogan GK-dim}, 
the socles of $X(\gamma_{10})$ and $X(\gamma_{11})$ are 
$\overline{X}(\gamma_{0}) \oplus \overline{X}(\gamma_{1})$. 
Therefore, the information of the composition factors of
$X(\gamma_{7})$ and $X(\gamma_{9})$ implies that 
the socle of $X(\gamma_{7})$ consists of $\overline{X}(\gamma_{0})$
alone and that of $X(\gamma_{9})$ is 
$\overline{X}(\gamma_{0}) \oplus \overline{X}(\gamma_{1})$. 
So by the parity condition, the socle filtrations of $X(\gamma_{7})$
and $X(\gamma_{9})$ are determined: 
\begin{align*}
& 
X(\gamma_{7}) 
\simeq 
\begin{matrix}
\overline{X}(\gamma_{7}) 
\\
\overline{X}(\gamma_{4}) 
\oplus 
\overline{X}(\gamma_{5}) 
\\
\overline{X}(\gamma_{0}) 
\end{matrix}, 
&
&
\mbox{and} 
&
&
X(\gamma_{9}) 
\simeq 
\begin{matrix}
\overline{X}(\gamma_{9}) 
\\
\overline{X}(\gamma_{4}) 
\oplus 
\overline{X}(\gamma_{5}) 
\oplus 
\overline{X}(\gamma_{6}) 
\\
\overline{X}(\gamma_{0}) 
\oplus 
\overline{X}(\gamma_{1}) 
\end{matrix}. 
\end{align*}
Bearing the followings in mind; 
(i) the above informations on the structures of 
$X(\gamma_{10})$ and $X(\gamma_{11})$, 
especially 
$X(\gamma_{7}), X(\gamma_{9}) \hookrightarrow X(\gamma_{10})$ and 
$X(\gamma_{9}) \hookrightarrow X(\gamma_{11})$, 
(ii) the informations of composition factors of $X(\gamma_{j})$ 
($j=10,11$), 
(iii) the images of Jacquet integrals $X(\gamma_{10}) \to \WhLmg$
and $X(\gamma_{11}) \to \WhLmg$ are 
compatible with \eqref{eq:Sp2 PSO(3,2)-1} 
and (iv) the parity condition, 
we know that the following diagrams are contained in the socle
filtration of $X(\gamma_{10})$ and $X(\gamma_{11})$ respectively: 
\begin{align*}
& 
\begin{matrix}
\overline{X}(\gamma_{10}) 
\\
\overline{X}(\gamma_{7}) 
\oplus 
\overline{X}(\gamma_{9}) 
\\
\overline{X}(\gamma_{4}) 
\oplus 
\overline{X}(\gamma_{5}) 
\\
\overline{X}(\gamma_{0}) 
\oplus 
\overline{X}(\gamma_{1}) 
\end{matrix},
&
& 
\begin{matrix}
\overline{X}(\gamma_{11}) 
\\
\overline{X}(\gamma_{2}) 
\oplus 
\overline{X}(\gamma_{3}) 
\oplus 
\overline{X}(\gamma_{9}) 
\\
\overline{X}(\gamma_{4}) 
\oplus 
\overline{X}(\gamma_{5}) 
\oplus 
\overline{X}(\gamma_{6}) 
\\
\overline{X}(\gamma_{0}) 
\oplus 
\overline{X}(\gamma_{1}) 
\end{matrix}.
\end{align*}

Consider the images of Jacquet integrals from $X(\gamma_{10})$ and 
$X(\gamma_{11})$ to $E(\overline{X}(\gamma_{0})) \subset \WhLmg$, 
which are non-zero on the $\overline{X}(\gamma_{0})$ in the socles of
them. 
We can check that the following diagram is contained in the socle 
filtration of $E(\overline{X}(\gamma_{0}))$: 
\begin{equation}
\label{eq:half of E(X_0)}
\begin{matrix}
\overline{X}(\gamma_{10}) 
\oplus 
\overline{X}(\gamma_{11}) 
\\
\overline{X}(\gamma_{1}) 
\oplus 
\overline{X}(\gamma_{2}) 
\oplus 
\overline{X}(\gamma_{9}) 
\oplus 
\overline{X}(\gamma_{7}) 
\\
\overline{X}(\gamma_{4}) 
\oplus 
\overline{X}(\gamma_{5}) 
\\
\overline{X}(\gamma_{0})
\end{matrix}. 
\end{equation}

Next, we use Theorem~\ref{theorem:1st duality theorem}.
In the Grothendieck group, 
\begin{align*}
E(\overline{X}(\gamma_{0})) 
&= \WhLmg - E(X(\gamma_{0'})) - E(X(\gamma_{0''})) 
= X(\gamma_{10}) + X(\gamma_{11}) 
\\
&= 
2 * \overline{X}(\gamma_{0}) 
+ 
2 * \overline{X}(\gamma_{1}) 
+ 
\overline{X}(\gamma_{2}) 
+ 
\overline{X}(\gamma_{3}) 
\\
& \quad 
+ 
3 * \overline{X}(\gamma_{4}) 
+ 
2 * \overline{X}(\gamma_{5}) 
+ 
2 * \overline{X}(\gamma_{6}) 
\\
& \quad 
+ 
\overline{X}(\gamma_{7}) 
+ 
\overline{X}(\gamma_{8}) 
+ 
2 * \overline{X}(\gamma_{9}) 
+ 
\overline{X}(\gamma_{10}) 
+ 
\overline{X}(\gamma_{11}). 
\end{align*}
Now, the multiplicity of the finite dimensional representation 
$\overline{X}(\gamma_{10})$ in $E(\overline{X}(\gamma_{0}))$ 
is one and it is self-adjoint. 
Then, by Theorem~\ref{theorem:1st duality theorem}, 
$\overline{X}(\gamma_{10})$ lies in the middle floor of 
$E(\overline{X}(\gamma_{0}))$. 
Moreover, note that the multiplicities of $\overline{X}(\gamma_{2})$
and $\overline{X}(\gamma_{3})$ in 
$E(\overline{X}(\gamma_{0}))$ are both one 
and $\overline{X}(\gamma_{2})^{c} = \overline{X}(\gamma_{3})$. 
Then by \eqref{eq:half of E(X_0)} and 
Theorem~\ref{theorem:1st duality theorem}, 
$\overline{X}(\gamma_{3})$ must be located above
$\overline{X}(\gamma_{11})$, so 
$E(\overline{X}(\gamma_{0}))^{c} 
\simeq 
E(\overline{X}(\gamma_{0}))$. 
Therefore, the following diagram is contained in the socle filtration
of $E(\overline{X}(\gamma_{0}))$: 
\[
\begin{matrix}
\overline{X}(\gamma_{1})
\\
\overline{X}(\gamma_{4}) 
\oplus 
\overline{X}(\gamma_{6}) 
\\
\overline{X}(\gamma_{0}) 
\oplus 
\overline{X}(\gamma_{3}) 
\oplus 
\overline{X}(\gamma_{9}) 
\oplus 
\overline{X}(\gamma_{8}) 
\\
\overline{X}(\gamma_{10}) 
\oplus 
\overline{X}(\gamma_{11}) 
\\
\overline{X}(\gamma_{1}) 
\oplus 
\overline{X}(\gamma_{2}) 
\oplus 
\overline{X}(\gamma_{9}) 
\oplus 
\overline{X}(\gamma_{7}) 
\\
\overline{X}(\gamma_{4}) 
\oplus 
\overline{X}(\gamma_{5}) 
\\
\overline{X}(\gamma_{0})
\end{matrix}. 
\]
Finally, compare this with the character 
of $E(\overline{X}(\gamma_{0}))$. 
Then one $\overline{X}(\gamma_{4})$, one $\overline{X}(\gamma_{5})$ 
and one $\overline{X}(\gamma_{6})$ are remaining. 
Since $\overline{X}(\gamma_{4})$ is self-dual, 
it must be located in the middle floor. 
By the parity condition, $\overline{X}(\gamma_{5})$ 
and $\overline{X}(\gamma_{6})$ must be located in the second, middle or
sixth floor, but by Lemma~\ref{lemma:Ext of DS and its Cayley transf}, 
they cannot be in the second floor. 
The self-duality forces them to be located in the middle floor. 

We obtain the socle filtration of 
$E(\overline{X}(\gamma_{0}))$: 
\begin{equation}
\label{eq:socle filtration of E(X(gamma_0)), PSO(3,2)}
E(\overline{X}(\gamma_{0})) 
\simeq 
\begin{matrix}
\overline{X}(\gamma_{1})
\\
\overline{X}(\gamma_{4}) 
\oplus 
\overline{X}(\gamma_{6}) 
\\
\overline{X}(\gamma_{0}) 
\oplus 
\overline{X}(\gamma_{3}) 
\oplus 
\overline{X}(\gamma_{9}) 
\oplus 
\overline{X}(\gamma_{8}) 
\\
\overline{X}(\gamma_{4}) 
\oplus 
\overline{X}(\gamma_{10}) 
\oplus 
\overline{X}(\gamma_{11}) 
\oplus 
\overline{X}(\gamma_{5}) 
\oplus 
\overline{X}(\gamma_{6}) 
\\
\overline{X}(\gamma_{1}) 
\oplus 
\overline{X}(\gamma_{2}) 
\oplus 
\overline{X}(\gamma_{9}) 
\oplus 
\overline{X}(\gamma_{7}) 
\\
\overline{X}(\gamma_{4}) 
\oplus 
\overline{X}(\gamma_{5}) 
\\
\overline{X}(\gamma_{0})
\end{matrix}. 
\end{equation}
The socle filtration of $E(\overline{X}(\gamma_{1}))$ is obtained by
turning this diagram upside down. 
\begin{remark}
\label{remark:self-duality in Sp(n,R) case}
In the above consideration, the facts that
$E(\overline{X}(\gamma_{0}))$ contains 
$\overline{X}(\gamma_{2})$ and $\overline{X}(\gamma_{3})$ with
multiplicity one and that 
$\overline{X}(\gamma_{2})^{c} \simeq
\overline{X}(\gamma_{3})$ 
forced 
$E(\overline{X}(\gamma_{0}))^{c} 
\simeq 
E(\overline{X}(\gamma_{0}))$. 
Just in the same way, we can show that, if $G = Sp(n,\R)$ (which does
not satisfy $G=G_{\max}$) and $X$ is a large irreducible module, then 
$E(X)^{c} \simeq E(X)$. Compare this with 
Theorems~\ref{theorem:1st duality theorem} and 
\ref{theorem:2nd duality theorem}. 
\end{remark}


\noindent
\textbf{II. $G=Sp(2,\R)$, block $PSO(4,1)$.} 

We use the notation of the block $PSO(3,2)$ case. 
Table~\ref{table:PSO(4,1)} is the data
on the regular characters in the block $PSO(4,1)$. 
\begin{table}[h]
\centering
\caption{Data on the regular characters in the block $PSO(4,1)$}
\label{table:PSO(4,1)}
\begin{tblr}{|c|c|c|c|c|c|c|c|c|}
\hline
\SetCell[r=2]{c} 
& \SetCell[r=2]{c} CSG & \SetCell[c=2]{c} simple root 
& & 
\SetCell[c=2]{c} cross action 
& & 
\SetCell[c=2]{c} Cayley transf.  
&
& \SetCell[r=2]{c} length 
\\ \hline
& & $\alpha_{1}$ & $\alpha_{2}$ 
& $s_{1} \times *$ & $s_{2} \times *$ & 
$\alpha_{1}$ & $\alpha_{2}$ & 
\\ \hline 
$\gamma_{0'}$ & $H_{J}$ & rn & $C^{+}$ & 
$\gamma_{0'}$ & $\gamma_{2'}$ & -- & -- & $1$
\\ \hline 
$\gamma_{1'}$ & $H_{J}$ & rn & $C^{+}$ & 
$\gamma_{1'}$ & $\gamma_{3'}$ & -- & -- & $1$
\\ \hline 
$\gamma_{2'}$ & $H_{J}$ & ncI & $C^{-}$ 
& $\gamma_{3'}$ & $\gamma_{0'}$ & $\gamma_{4'}$ & -- 
& $2$
\\ \hline 
$\gamma_{3'}$ & $H_{J}$ & ncI & $C^{-}$ 
& $\gamma_{2'}$ & $\gamma_{1'}$ & $\gamma_{4'}$ & -- 
& $2$
\\ \hline 
$\gamma_{4'}$ & $H_{s}$ & rI & rn 
& $\gamma_{4'}$ & $\gamma_{4'}$ & $\gamma_{2'}$, $\gamma_{3'}$ 
& -- & $3$
\\ \hline 
\end{tblr}
\end{table}

In the Grothendieck group, 
\begin{align*}
&
X(\gamma_{i'}) = \overline{X}(\gamma_{i'}), 
\quad i=0, 1, 
\qquad
X(\gamma_{j'}) 
= 
\overline{X}(\gamma_{(j-2)'}) + \overline{X}(\gamma_{j'}), 
\quad j=2,3 
\\
&
X(\gamma_{4'}) 
= 
\overline{X}(\gamma_{0'}) + \overline{X}(\gamma_{1'})
+
\overline{X}(\gamma_{2'}) + \overline{X}(\gamma_{3'}) 
+ \overline{X}(\gamma_{4'}). 
\end{align*}
Note that irreducible large modules are $\overline{X}(\gamma_{0'})$
and $\overline{X}(\gamma_{1'})$. 

The socle filtration of $E(\overline{X}(\gamma_{0'}))$ is given by 
\begin{align}
&
E(\overline{X}(\gamma_{0'})) 
\simeq 
\begin{matrix}
\overline{X}(\gamma_{1'})
\\
\overline{X}(\gamma_{3'}) 
\\
\overline{X}(\gamma_{4'})
\\
\overline{X}(\gamma_{2'})
\\
\overline{X}(\gamma_{0'})
\end{matrix}, 
\label{eq:socle filtration of E(X(gamma_0')), PSO(4,1)}
\end{align}
The socle filtration of $E(\overline{X}(\gamma_{1'}))$ is obtained by
turning this diagram upside down.


\

\noindent
\textbf{III. $G = SL(3,\R)$, block $PSU(2,1)$.}

In the group $SL(3,\R)$,
there are two conjugacy classes of Cartan subgroups. 
One is fundamental and the other is split. 
We denote them by $H_{f}$ and $H_{s}$, respectively. 
There are three $K$-conjugacy classes of regular characters of $H_{f}$, 
which we call $\gamma_{0}$, $\gamma_{1}$ and $\gamma_{2}$. 
There are four $K$-conjugacy classes of regular characters of $H_{s}$, 
which we call $\gamma_{3}$, $\gamma_{4}$, $\gamma_{5}$ and 
$\gamma_{0'}$. 
The regular characters $\gamma_{j}$ ($j=0,1,\dots,5$) are contained in
the block $PSU(2,1)$ and the block $PSU(3)$ consists of
$\gamma_{0'}$. 
$X(\gamma_{0'})$ is an irreducible principal series. 
Table~\ref{table:PSU(2,1)} is the data
on the regular characters in the block $PSU(2,1)$. 
\begin{table}[h]
\centering
\caption{Data on the regular characters in the block $PSU(2,1)$}
\label{table:PSU(2,1)}
\begin{tblr}{|c|c|c|c|c|c|c|c|c|}
\hline
\SetCell[r=2]{c} 
& \SetCell[r=2]{c} CSG & \SetCell[c=2]{c} simple root 
& & 
\SetCell[c=2]{c} cross action 
& & 
\SetCell[c=2]{c} Cayley transf.  
&
& \SetCell[r=2]{c} length 
\\ \hline
& & $\alpha_{1}$ & $\alpha_{2}$ 
& $s_{1} \times *$ & $s_{2} \times *$ & 
$\alpha_{1}$ & $\alpha_{2}$ & 
\\ \hline 
$\gamma_{0}$ & $H_{f}$ & $C^{+}$ & $C^{+}$ & 
$\gamma_{2}$ & $\gamma_{1}$ & -- & -- & $0$
\\ \hline 
$\gamma_{1}$ & $H_{f}$ & ncII & $C^{-}$ 
& $\gamma_{1}$ & $\gamma_{0}$ & $\gamma_{3}$, $\gamma_{4}$ & -- 
& $1$
\\ \hline 
$\gamma_{2}$ & $H_{f}$ & $C^{-}$ & ncII
& $\gamma_{0}$ & $\gamma_{2}$ & -- & $\gamma_{3}$, $\gamma_{5}$ 
& $1$
\\ \hline 
$\gamma_{3}$ & $H_{s}$ & rII & rII
& $\gamma_{4}$ & $\gamma_{5}$ & $\gamma_{1}$ & $\gamma_{2}$ 
& $2$
\\ \hline 
$\gamma_{4}$ & $H_{s}$ & rII & rn 
& $\gamma_{3}$ & $\gamma_{4}$ & $\gamma_{1}$ &  -- 
& $2$
\\ \hline 
$\gamma_{5}$ & $H_{s}$ & rn & rII
& $\gamma_{5}$ & $\gamma_{3}$ & -- & $\gamma_{2}$ 
& $2$
\\ \hline 
\end{tblr}
\end{table}

In the Grothendieck group, 
\begin{align*}
&
X(\gamma_{0}) = \overline{X}(\gamma_{0}), 
\qquad
X(\gamma_{1}) = \overline{X}(\gamma_{0}) + \overline{X}(\gamma_{1}), 
\qquad
X(\gamma_{2}) = \overline{X}(\gamma_{0}) + \overline{X}(\gamma_{2}), 
\\
&
X(\gamma_{3}) 
= 
\overline{X}(\gamma_{0})+\overline{X}(\gamma_{1})
+\overline{X}(\gamma_{2})+\overline{X}(\gamma_{3}), 
\\
&
X(\gamma_{4})
= 
\overline{X}(\gamma_{0})+\overline{X}(\gamma_{1})+\overline{X}(\gamma_{4}), 
\qquad
X(\gamma_{5})
= 
\overline{X}(\gamma_{0})+\overline{X}(\gamma_{2})+\overline{X}(\gamma_{5}).
\end{align*}
Note that large irreducible modules are $\overline{X}(\gamma_{0})$ and 
$\overline{X}(\gamma_{0'})$. 
Since the block $PSU(3)$ consists of $\overline{X}(\gamma_{0'})$, 
its injective envelope is isomorphic to itself; 
$E(\overline{X}(\gamma_{0'})) \simeq \overline{X}(\gamma_{0'})$. 

The socle filtration of $E(\overline{X}(\gamma_{0}))$ is given by 
\begin{equation}
\label{eq:socle filtration of E(X(gamma_0)), PSU(2,1)}
E(\overline{X}(\gamma_{0})) 
\simeq 
\begin{matrix}
\overline{X}(\gamma_{0})
\\
\overline{X}(\gamma_{1}) 
\oplus \overline{X}(\gamma_{2}) 
\\
\overline{X}(\gamma_{0})
\oplus 
\overline{X}(\gamma_{3}) 
\oplus \overline{X}(\gamma_{4}) \oplus \overline{X}(\gamma_{5}) 
\\
\overline{X}(\gamma_{1})
\oplus \overline{X}(\gamma_{2}) 
\\
\overline{X}(\gamma_{0})
\end{matrix}. 
\end{equation}


\noindent
\textbf{IV. $G = G_{2}$ (split, linear).}

In the group split $G_{2}$, 
there are four conjugacy classes of Cartan subgroups. 
The fundamental Cartan subgroup $H_{f}$ is a compact Cartan subgroup. 
Between $H_{f}$ and the split Cartan subgroup $H_{s}$, 
there are two Cartan subgroups. 
One is the Cayley transform of the compact Cartan subgroup 
through a short non-compact imaginary root. 
We call it $H_{1} = T_{1} A_{1}$. 
The other intermediate Cartan subgroup is the Cayley transform of 
the compact Cartan subgroup through a long non-compact imaginary
root. 
We call it $H_{2} = T_{2} A_{2}$. 

There are three $K$-conjugacy classes of regular characters of $H_{f}$, 
which we call $\gamma_{0}$, $\gamma_{1}$ and $\gamma_{2}$. 
There are three $K$-conjugacy classes of regular characters 
of $H_{1}$, 
which we call $\gamma_{3}$, $\gamma_{6}$ and $\gamma_{7}$. 
There are three $K$-conjugacy classes of regular characters 
of $H_{2}$, 
which we call $\gamma_{4}$, $\gamma_{5}$ and $\gamma_{8}$. 
There are four $K$-conjugacy classes of regular characters 
of $H_{s}$, 
which we call $\gamma_{9}$, $\gamma_{10}$, $\gamma_{11}$ and 
$\gamma_{0'}$. 
$\gamma_{0'}$ is contained in the block $G_{2}$(compact) 
and other regular characters are all contained in the block 
$G_{2}$(split). 
Among them, the regular characters which correspond to large
irreducible modules are $\gamma_{0}$ and $\gamma_{0'}$. 
$\overline{X}(\gamma_{0'})$
is an irreducible principal series, so 
$E(\overline{X}(\gamma_{0'})) \simeq \overline{X}(\gamma_{0'})$. 

Let $\alpha_{1}$ be the short simple root and $\alpha_{2}$ the
long simple root. 
Table~\ref{table:G2} is the data
on the regular characters in the block $G_{2}$(split). 
\begin{table}[h]
\centering
\caption{Data on the regular characters in the block $G_{2}$(split)}
\label{table:G2}
\begin{tblr}{|c|c|c|c|c|c|c|c|c|}
\hline
\SetCell[r=2]{c} 
& \SetCell[r=2]{c} CSG & \SetCell[c=2]{c} simple root 
& & 
\SetCell[c=2]{c} cross action 
& & 
\SetCell[c=2]{c} Cayley transf.  
&
& \SetCell[r=2]{c} length 
\\ \hline
& & $\alpha_{1}$ & $\alpha_{2}$ 
& $s_{1} \times *$ & $s_{2} \times *$ & 
$\alpha_{1}$ & $\alpha_{2}$ & 
\\ \hline 
$\gamma_{0}$ & $H_{f}$ & ncI & ncI & 
$\gamma_{1}$ & $\gamma_{2}$ & $\gamma_{3}$ & $\gamma_{4}$ & $0$
\\ \hline 
$\gamma_{1}$ & $H_{f}$ & ncI & cpt & 
$\gamma_{0}$ & $\gamma_{1}$ & $\gamma_{3}$ & -- & $0$
\\ \hline 
$\gamma_{2}$ & $H_{f}$ & cpt & ncI 
& $\gamma_{2}$ & $\gamma_{0}$ & -- & $\gamma_{4}$ 
& $0$
\\ \hline 
$\gamma_{3}$ & $H_{1}$ & rI & $C^{+}$ 
& $\gamma_{3}$ & $\gamma_{6}$ & $\gamma_{0}$, $\gamma_{1}$ & -- 
& $1$
\\ \hline 
$\gamma_{4}$ & $H_{2}$ & $C^{+}$ & rI 
& $\gamma_{5}$ & $\gamma_{4}$ & -- & $\gamma_{0}$, $\gamma_{2}$ 
& $1$
\\ \hline 
$\gamma_{5}$ & $H_{2}$ & $C^{-}$ & $C^{+}$ 
& $\gamma_{4}$ & $\gamma_{8}$ & -- & --
& $2$
\\ \hline 
$\gamma_{6}$ & $H_{1}$ & $C^{+}$ & $C^{-}$ 
& $\gamma_{7}$ & $\gamma_{3}$ & -- & --
& $2$
\\ \hline 
$\gamma_{7}$ & $H_{1}$ & $C^{-}$ & ncII 
& $\gamma_{6}$ & $\gamma_{7}$ & -- & $\gamma_{9}$, $\gamma_{11}$ 
& $3$
\\ \hline 
$\gamma_{8}$ & $H_{2}$ & ncII & $C^{-}$ 
& $\gamma_{8}$ & $\gamma_{5}$ & $\gamma_{9}$, $\gamma_{10}$ & --
& $3$
\\ \hline 
$\gamma_{9}$ & $H_{s}$ & rII & rII 
& $\gamma_{10}$ & $\gamma_{11}$ & $\gamma_{8}$ & $\gamma_{7}$ 
& $4$
\\ \hline 
$\gamma_{10}$ & $H_{s}$ & rII & rn 
& $\gamma_{9}$ & $\gamma_{10}$ & $\gamma_{8}$ & -- 
& $4$
\\ \hline 
$\gamma_{11}$ & $H_{s}$ & rn & rII
& $\gamma_{11}$ & $\gamma_{9}$ & -- & $\gamma_{7}$ 
& $4$
\\ \hline 
\end{tblr}
\end{table}

In the Grothendieck group, 
\begin{align*}
&
X(\gamma_{j}) = \overline{X}(\gamma_{j}), \ (j=0,1,2)
\\
&
X(\gamma_{3}) 
= \overline{X}(\gamma_{0}) + \overline{X}(\gamma_{1}) 
+ \overline{X}(\gamma_{3}), 
\qquad
X(\gamma_{4}) 
= \overline{X}(\gamma_{0}) + \overline{X}(\gamma_{2}) 
+ \overline{X}(\gamma_{4}), 
\\
&
X(\gamma_{5}) 
= 
\overline{X}(\gamma_{0})+\overline{X}(\gamma_{3})
+\overline{X}(\gamma_{4})+\overline{X}(\gamma_{5}), 
\\
&
X(\gamma_{6}) 
= 
\overline{X}(\gamma_{0})+\overline{X}(\gamma_{3})
+\overline{X}(\gamma_{4})+\overline{X}(\gamma_{6}), 
\\
&
X(\gamma_{7})
= 
\sum_{j=0, \not=1}^{7} \overline{X}(\gamma_{j}), 
\quad
X(\gamma_{8})
= 
\sum_{j=0, \not= 2, 7}^{8}\overline{X}(\gamma_{j}),
\quad 
X(\gamma_{9}) 
= 
\sum_{j=0}^{9} \overline{X}(\gamma_{j}), 
\\
& 
X(\gamma_{10}) 
= 
\overline{X}(\gamma_{0})+\overline{X}(\gamma_{1})
+2*\overline{X}(\gamma_{3})
+
\sum_{j=4}^{6} \overline{X}(\gamma_{j})
+\overline{X}(\gamma_{8}) 
+\overline{X}(\gamma_{10}),
\\
& 
X(\gamma_{11}) 
= 
\overline{X}(\gamma_{0})+\overline{X}(\gamma_{2})
+\overline{X}(\gamma_{3})+2*\overline{X}(\gamma_{4})
+
\sum_{j=5}^{7} \overline{X}(\gamma_{j})
+\overline{X}(\gamma_{11}). 
\end{align*} 
The socle filtration of $E(\overline{X}(\gamma_{0}))$ is given by 
\begin{equation}
\label{eq:socle filtration of E(X(gamma_0)), G2(split)}
E(\overline{X}(\gamma_{0})) 
\simeq 
\begin{matrix}
\overline{X}(\gamma_{0})
\\
\overline{X}(\gamma_{3}) 
\oplus \overline{X}(\gamma_{4}) 
\\
\overline{X}(\gamma_{1})
\oplus \overline{X}(\gamma_{2}) 
\oplus 
\overline{X}(\gamma_{5})
\oplus 
\overline{X}(\gamma_{6})
\\
\overline{X}(\gamma_{3}) 
\oplus 
\overline{X}(\gamma_{4}) 
\oplus 
\overline{X}(\gamma_{7}) 
\oplus 
\overline{X}(\gamma_{8}) 
\\
\overline{X}(\gamma_{0})
\oplus 
\overline{X}(\gamma_{5}) 
\oplus 
\overline{X}(\gamma_{6}) 
\oplus 
\overline{X}(\gamma_{9}) 
\oplus 
\overline{X}(\gamma_{10}) 
\oplus 
\overline{X}(\gamma_{11}) 
\\
\overline{X}(\gamma_{3}) 
\oplus 
\overline{X}(\gamma_{4}) 
\oplus 
\overline{X}(\gamma_{7}) 
\oplus 
\overline{X}(\gamma_{8}) 
\\
\overline{X}(\gamma_{1})
\oplus \overline{X}(\gamma_{2}) 
\oplus 
\overline{X}(\gamma_{5})
\oplus 
\overline{X}(\gamma_{6})
\\
\overline{X}(\gamma_{3}) 
\oplus \overline{X}(\gamma_{4}) 
\\
\overline{X}(\gamma_{0})
\end{matrix}. 
\end{equation}



\begin{thebibliography}{99} 


\bibitem{van den Ban} 
van den Ban, E. P.: 
Uniform temperedness of Whittaker integrals for a real reductive
group. 
Preprint. 



\bibitem{Casselman}
Casselman, W.: 
Canonical extensions of Harish-Chandra modules to representations of G.
Canad. J. Math. \textbf{41} (1989), no. 3, 385--438. 
MR1013462 (90j:22013)

\bibitem{CHM} Casselman,W. ; Hecht, H. ; Mili\v{c}i\'{c}, D.  : 
Bruhat filtrations and Whittaker vectors for real groups. 
Proceedings of Symposia in Pure Mathematics 
\textbf{68} (2000), 151--190. MR1767896 (2002b:22023)

\bibitem{GW} Goodman, R.; Wallach, N. R.:
Whittaker Vectors and Conical Vectors. 
J. Funct. Analysis \textbf{39} (1980), 199--279. 
MR0597811 (82i:22018)

\bibitem{HTY} 
Hashimoto, N.; Taniguchi, K.; Yamanaka, G.: 
The socle filtrations of principal series representations 
of $\mathrm{SL}(3,\mathbf{R})$ and $\mathrm{Sp}(2,\mathbf{R})$. 
Tokyo J. Math. \textbf{47} (2024), no. 1, 189--243. 
MR4787891

\bibitem{Jac} Jacquet, H.: 
Fonctions de Whittaker associ\'{e}es aux groupes de Chevalley. 
Bull. Soc. Math. France \textbf{95} (1967), 243--309. 
MR0271275 (42 \#6158)

\bibitem{Jan} Jantzen, J. C. : 
{\it Representations of algebraic groups}.
Second edition
Math. Surveys Monogr., 107
American Mathematical Society, Providence, RI, 2003. xiv+576 pp.
MR2015057 (2004h:20061)

\bibitem{Kashiwara-Oshima} 
Kashiwara, M. ; Oshima, T. : 
Systems of differential equations with regular singularities and their
boundary value problems. 
Ann. of Math. \textbf{106} (1977), 145--200. 
MR0482870 (58 \#2914)

\bibitem{KV} 
Knapp, W. A. ; Vogan Jr, D. A. : 
{\it Cohomological Induction and Unitary Representations.} 
Princeton Mathematical Series, 45. 
 Princeton University Press, Princeton, NJ, 1995, xvii +948 pp. 
MR1330919 (96c:22023)

\bibitem{Kos} Kostant, B. : 
On Whittaker Vectors and Representation Theory.
Invent. Math. \textbf{48} (1978), no. 2, 101--184.
MR507800 (80b:22020)

\bibitem{Lynch} 
Lynch, T. E.: 
Generalized Whittaker vectors and representation theory. 
Thesis, MIT, 1979. 


\bibitem{M0}
Matumoto, H. : 
Boundary value problems for Whittaker functions on real split
semisimple Lie groups. 
Duke Math. J. \textbf{53} (1986), no. 3, 635--676. 
MR0860664 (88b:22010)

\bibitem{M1}
Matumoto, H. : 
Whittaker vectors and the Goodman-Wallach operators, 
Acta math. \textbf{161} (1988), 183--241. 
MR0971796 (90d:22018)

\bibitem{M2}
Matumoto, H. :
$C^{-\infty}$-Whittaker vectors corresponding to a principal nilpotent
orbit of a real reductive linear Lie group, and wave front set, 
Compositio Math. \textbf{82} (1992), 189--244. 
MR1157939 (93c:22026)

\bibitem{Oshima ASPM1984}
Oshima, T. : 
Boundary value problems for systems of linear partial differential
equations with regular singularities. 
Group representations and systems of differential equations 
(Tokyo, 1982), 391--432, Adv. Stud. Pure Math. \textbf{4}, North-Holland,
Amsterdam, 1984. MR0810637 (87c:58121)

\bibitem{OS}
Oshima, T.; Sekiguchi, J.: Eigenspaces of invariant differential
operators on an affine symmetric spaces. 
Invent. Math. \textbf{57} (1980), 1--81. 
MR0564184 (81k:43014)

\bibitem{Speh-Vogan}
Speh, B.; Vogan, D. A. Jr.:
Reducibility of generalized principal series representations.
Acta Math. \textbf{145} (1980), no. 3-4, 227--299. 
MR0590291 (82c:22018)

\bibitem{Taniguchi 2013} 
Taniguchi, K.: 
On the Composition Series of the Standard Whittaker $\brgK$-modules. 
Trans. A.M.S. \textbf{365} (2013), no. 7, 3899--3922. 
MR3042608

\bibitem{Taniguchi 2015}
Taniguchi, K.: 
Socle filtrations of the standard Whittaker $\brgK$-modules of
$\mathrm{Spin}(r,1)$. 
Kyoto J. Math. \textbf{55} (2015), no. 1, 43--61.
MR3323527

\bibitem{Vogan GK-dim} 
Vogan, D. A.: 
Gelfand-Kirillov Dimension for Harish-Chandra Modules. 
Inventiones mathematicae \textbf{48}(1978), 75--98. 
MR0506503 (58 \#22205)

\bibitem{Vogan Green} 
Vogan, D. A.: 
{\it Representations of real reductive Lie groups}. 
Progress in Mathematics, 15. Birkha\"{u}ser, Boston, Mass., 1981.
MR0632407 (83c:22022)

\bibitem{Vogan PM40} 
Vogan, D. A.: 
The Kazhdan-Lusztig conjecture for real reductive groups. 
Representation theory of reductive groups 
(Park City, Utah, 1982), 
223--264.
Progress in Mathematics, 40. 
Birkha\"{u}ser, Boston, Mass., 1983. 
MR0733817 (85g:22028)

\bibitem{W} 
Wallach, N. R. : 
Asymptotic expansions of generalized matrix entries of representations
of real reductive groups, 
Lie group representations, I, 287--369, Lecture Notes in
Math., \textbf{1024}, Springer Verlag, Berlin, 1983. 
MR0727854 (85g:22029)

\bibitem{Wallach real reductive I} 
Wallach, N. R. : 
{\it Real Reductive Groups I}. 
Pure Appl. Math., 132. 
Academic Press, Inc., Boston, MA, 1988. xx+412 pp. 
MR0929683 (89i:22029)

\bibitem{Wallach real reductive II} 
Wallach, N. R. : 
{\it Real Reductive Groups II}. 
Pure Appl. Math., 132-II. 
Academic Press, Inc., Boston, MA, 1992. xiv+454 pp.
MR1170566 (93m:22018)

\end{thebibliography}
\end{document}